\documentclass
[
    fontsize = 11 pt,
    american,
    captions = tableheading,
    numbers = noenddot,
    abstracton,
    paper = a4paper,
    DIV=13,
    left=1in,
    right=1in,
    top=1in,
    bottom=1in,
]
{scrartcl}

\usepackage[T1]{fontenc}
\usepackage[utf8]{inputenc}
\usepackage
[
    babel = true, 
]
{microtype}           
\usepackage{csquotes} 

\usepackage[dvipsnames]{xcolor} 

\definecolor{stroke1}{HTML}{2574A9} 

\usepackage{amsmath}
\usepackage{amssymb}
\usepackage{amsthm}
\usepackage{thmtools}
\usepackage{mathtools}
\usepackage{thm-restate}
\usepackage{dsfont}        
\usepackage{braceMnSymbol} 

\usepackage
[
    ttscale = 0.85, 
]
{libertine} 
\usepackage
[
    libertine,    
    slantedGreek, 
    vvarbb,       
    libaltvw,     
]
{newtxmath} 
\usepackage{url} 
\usepackage{bm}  

\usepackage{graphicx} 
\usepackage
[
    subrefformat = simple, 
    labelformat = simple,  
]
{subcaption}         

\usepackage{array}     
\usepackage{booktabs}  
\usepackage{longtable} 
\usepackage{pdflscape} 

\usepackage[shortlabels]{enumitem} 

\usepackage
[
    ruled,         
    vlined,        
    linesnumbered, 
]
{algorithm2e} 

\usepackage{xspace}   
\usepackage
[
    shortcuts, 
]
{extdash}             
\usepackage{setspace} 

\usepackage{xparse}    
\usepackage{footnote}  
\usepackage{afterpage} 

\usepackage
[
    sortcites,              
    style = alphabetic,     
    defernumbers,           
    safeinputenc,           
    backref = true,         
    backrefstyle = three,   
    hyperref = true,        
    maxbibnames = 99,       
    mincitenames = 1,       
    maxcitenames = 3,       
    minalphanames = 3,       
]
{biblatex} 

\DefineBibliographyStrings{english}%
{%
    backrefpage  = {\lowercase{s}ee page}, 
    backrefpages = {\lowercase{s}ee pages} 
}

\usepackage{multirow}
\usepackage{caption}
\usepackage
[
    bookmarks = true,                 
    bookmarksopen = false,            
    bookmarksnumbered = true,         
    pdfstartpage = 1,                 
    colorlinks = true,                
    allcolors = stroke1,              
]
{hyperref} 

\usepackage
[
    noabbrev,   
    nameinlink, 
]
{cleveref} 

\date{}

\makeatletter
    \def\IfEmptyTF#1%
    {%
        \if\relax\detokenize{#1}\relax%
            \expandafter\@firstoftwo%
        \else%
            \expandafter\@secondoftwo%
        \fi%
    }
\makeatother

\NewDocumentCommand{\mathOrText}{m}
{%
    \ensuremath{#1}\xspace%
}

\let\originalleft\left
\let\originalright\right
\renewcommand{\left}{\mathopen{}\mathclose\bgroup\originalleft}
\renewcommand{\right}{\aftergroup\egroup\originalright}

\makeatletter
    \DeclareRobustCommand{\bfseries}%
    {%
        \not@math@alphabet\bfseries\mathbf%
        \fontseries\bfdefault\selectfont%
        \boldmath%
    }

\xspaceaddexceptions{]\}}

\allowdisplaybreaks

\crefname{ineq}{inequality}{inequalities}
\creflabelformat{ineq}{#2{\upshape(#1)}#3}

\crefname{term}{term}{terms}
\creflabelformat{term}{#2{\upshape(#1)}#3}

\crefname{cond}{condition}{conditions}
\creflabelformat{cond}{#2{\upshape(#1)}#3}

\crefname{assume}{assumption}{assumptions}
\creflabelformat{assume}{#2{\upshape(#1)}#3}

\let\oldfootnote\footnote

\newlength{\spaceBeforeFootnote} 
\newlength{\spaceAfterFootnote}  

\RenewDocumentCommand{\footnote}{o o o m}%
{%
    \IfNoValueTF{#1}%
    {%
        \oldfootnote{#4}%
    }%
    {%
        \setlength{\spaceBeforeFootnote}{\IfEmptyTF{#1}{0}{#1} em}%
        \IfNoValueTF{#2}%
        {%
            \hspace*{\spaceBeforeFootnote}\oldfootnote{#4}%
        }%
        {%
            \setlength{\spaceAfterFootnote}{\IfEmptyTF{#2}{0}{#2} em}%
            \hspace*{\spaceBeforeFootnote}\IfNoValueTF{#3}{\oldfootnote{#4}}{\oldfootnote[#3]{#4}}\hspace*{\spaceAfterFootnote}%
        }%
    }%
}

\makesavenoteenv{figure}
\makesavenoteenv{table}
\makesavenoteenv{tabular}

\declaretheoremstyle
[
   	spaceabove = \topsep,
   	spacebelow = \topsep,
   	headfont = \bfseries,
   	headformat = \textcolor{stroke1}{$\blacktriangleright$} \NAME~\NUMBER \NOTE,
   	notefont = \bfseries,
   	notebraces = {(}{)},
   	bodyfont = \normalfont,
   	postheadspace = 0.5 em,
   	qed = \textcolor{stroke1}{\bfseries$\blacktriangleleft$},
]
{myTheoremStyle}

\declaretheorem
[
   	style = myTheoremStyle,
   	name = Lemma,
    sharenumber = conjecture,
]
{lemma}
\declaretheorem
[
   	style = myTheoremStyle,
   	name = Corollary,
    sharenumber = conjecture,
]
{corollary}
\declaretheorem
[
   	style = myTheoremStyle,
   	name = Theorem,
    sharenumber = conjecture,
]
{theorem}
\declaretheorem
[
   	style = myTheoremStyle,
   	name = Definition,
    sharenumber = conjecture,
]
{definition}

\NewDocumentCommand{\functionTemplate}{m m m m o}%
{%
    \IfNoValueTF{#5}%
    {%
        \mathOrText{#1\left#2{#4}\right#3}%
    }%
    {%
        \mathOrText{#1#5#2{#4}#5#3}%
    }%
}

\newcommand*{\leftBracketType}{(}
\newcommand*{\rightBracketType}{)}

\NewDocumentCommand{\createFunction}{m m o o}%
{%
    \renewcommand*{\leftBracketType}{\IfNoValueTF{#3}{(}{#3}}%
    \renewcommand*{\rightBracketType}{\IfNoValueTF{#4}{)}{#4}}%
    \NewDocumentCommand{#1}{o o}%
    {%
        \IfNoValueTF{##1}%
        {%
            \mathOrText{#2}%
        }%
        {%
            \functionTemplate{#2}{\leftBracketType}{\rightBracketType}{##1}[##2]%
        }%
    }%
}

\DeclareDocumentCommand{\probabilisticFunctionTemplate}{m m O{} o}
{%
    \functionTemplate{#1}%
    {\lbrack}%
    {\rbrack}%
    {#2\IfEmptyTF{#3}{}{\ \IfNoValueTF{#4}{\left}{#4}\vert\ \vphantom{#2}#3\IfNoValueTF{#4}{\right.}{}}}%
    [#4]%
}

\newcommand*{\N}{\mathOrText{\mathds{N}}}

\newcommand*{\R}{\mathOrText{\mathds{R}}}

\newcommand*{\indicatorFunctionSymbol}{\mathds{1}}

\RenewDocumentCommand{\Pr}{m O{} o}%
{%
    \probabilisticFunctionTemplate{\mathrm{Pr}}{#1}[#2][#3]%
}

\NewDocumentCommand{\E}{m O{} o}%
{%
    \probabilisticFunctionTemplate{\mathrm{E}}{#1}[#2][#3]%
}

\NewDocumentCommand{\Var}{m O{} o}%
{%
    \probabilisticFunctionTemplate{\mathrm{Var}}{#1}[#2][#3]%
}

\DeclareDocumentCommand{\bigO}{m o}%
{%
    \functionTemplate{\mathrm{O}}{(}{)}{#1}[#2]%
}

\DeclareDocumentCommand{\smallO}{m o}%
{%
    \functionTemplate{\mathrm{o}}{(}{)}{#1}[#2]%
}

\DeclareDocumentCommand{\bigTheta}{m o}%
{%
    \functionTemplate{\upTheta}{(}{)}{#1}[#2]%
}

\DeclareDocumentCommand{\bigOmega}{m o}%
{%
    \functionTemplate{\upOmega}{(}{)}{#1}[#2]%
}

\DeclareDocumentCommand{\smallOmega}{m o}%
{%
    \functionTemplate{\upomega}{(}{)}{#1}[#2]%
}

\DeclareDocumentCommand{\eulerE}{o}%
{%
    \mathOrText{\mathrm{e}\IfNoValueTF{#1}{}{^{#1}}}%
}

\DeclareDocumentCommand{\poly}{m o}%
{%
    \functionTemplate{\mathrm{poly}}{(}{)}{#1}[#2]%
}

\createFunction{\id}{\mathrm{id}}

\NewDocumentCommand{\ind}{m o o}%
{%
    \IfNoValueTF{#2}%
    {%
        \mathOrText{\indicatorFunctionSymbol_{#1}}%
    }%
    {%
        \functionTemplate{\indicatorFunctionSymbol_{#1}}{(}{)}{#2}[#3]%
    }%
}

\DeclareDocumentCommand{\dom}{m o}%
{%
    \functionTemplate{\mathrm{dom}}{(}{)}{#1}[#2]%
}

\DeclareDocumentCommand{\rng}{m o}%
{%
    \functionTemplate{\mathrm{rng}}{(}{)}{#1}[#2]%
}

\DeclareDocumentCommand{\d}{o}%
{%
    \mathrm{d}\IfNoValueTF{#1}{}{^{#1}}%
}

\DeclareDocumentCommand{\set}{m m o}%
{%
    \mathOrText{\IfNoValueTF{#3}{\left}{#3}\{#1\ \IfNoValueTF{#3}{\left}{#3}\vert\ \vphantom{#1}#2\IfNoValueTF{#3}{\right.}{}\IfNoValueTF{#3}{\right}{#3}\}}%
}

\newcommand*{\indicator}[1]{\mathOrText{\mathbf{1}{\{#1\}}}}

\DeclareDocumentCommand{\randomProcess}{m o}
{%
    \mathOrText{X^{(#1)}\IfNoValueTF{#2}{}{_{#2}}}%
}

\DeclareDocumentCommand{\transformedProcess}{o}
{%
    \mathOrText{Y\IfNoValueTF{#1}{}{_{#1}}}%
}

\DeclareDocumentCommand{\filtration}{o}
{%
    \mathOrText{\mathcal{F}\IfNoValueTF{#1}{}{_{#1}}}%
}

\newcommand*{\numberOfVertices}{\mathOrText{n}}

\newcommand*{\contactProcess}{\mathOrText{C}}

\newcommand*{\timeContinuous}[1]{\mathOrText{\tau_{#1}}}

\DeclareDocumentCommand{\lyapunovHelper}{o}
{%
    \mathOrText{f\IfNoValueTF{#1}{}{(#1)}}%
}

\DeclareDocumentCommand{\lyapunovFunction}{m o}
{%
    \mathOrText{F\IfNoValueTF{#2}{}{(#2)}}%
}

\DeclareDocumentCommand{\potentialFunctionIMinusR}{o}
{%
    \mathOrText{H\IfNoValueTF{#1}{}{(#1)}}%
}

\newcommand*{\IRIR}{\mathOrText{\text{IRIR}}}
\newcommand*{\Ia}[1]{\mathOrText{I_{1,#1}}}
\newcommand*{\Ib}[1]{\mathOrText{I_{2,#1}}}
\newcommand*{\Ra}[1]{\mathOrText{R_{1,#1}}}
\newcommand*{\Rb}[1]{\mathOrText{R_{2,#1}}}
\newcommand*{\da}[1]{\mathOrText{\delta_{1,#1}}}
\newcommand*{\db}[1]{\mathOrText{\delta_{2,#1}}}
\newcommand*{\la}{\mathOrText{\lambda_1}}
\newcommand*{\lb}{\mathOrText{\lambda_2}}
\newcommand*{\lae}{\mathOrText{\lambda_1^*}}
\newcommand*{\lbe}{\mathOrText{\lambda_2^*}}
\newcommand*{\ca}{\mathOrText{c_1}}
\newcommand*{\cb}{\mathOrText{c_2}}
\newcommand*{\ra}{\mathOrText{\varrho_1}}
\newcommand*{\rb}{\mathOrText{\varrho_2}}
\newcommand*{\Iae}{\mathOrText{I_1^*}}
\newcommand*{\Ibe}{\mathOrText{I_2^*}}
\newcommand*{\Rae}{\mathOrText{R_1^*}}
\newcommand*{\Rbe}{\mathOrText{R_2^*}}
\newcommand*{\xa}{\mathOrText{x_1}}
\newcommand*{\xb}{\mathOrText{x_2}}
\newcommand*{\ya}{\mathOrText{y_1}}
\newcommand*{\yb}{\mathOrText{y_2}}
\newcommand*{\Lyapunov}{\mathOrText{F}}
\newcommand*{\Lyapunovb}{\mathOrText{K}}
\newcommand*{\Lyapunovc}{\mathOrText{L}}
\newcommand*{\density}{\mathOrText{p}}

\newcommand*{\erdosRenyi}{Erd\H{o}s--Rényi\xspace}

\usepackage[utf8]{inputenc}

\title{\LARGE A Tight Epidemic Threshold for Competing Stochastic Infection Processes with Mutually Exclusive Immunity}
\author{%
    Nicolas Klodt$^{*}$ \and Martin~S. Krejca$^{\dagger}$
}

\publishers
 {
  	\footnotesize $^{*}$ Hasso Plattner Institute, University of Potsdam, Potsdam, Germany\\
  	nicolas.klodt@hpi.de \\

  	$^{\dagger}$ LIX, CNRS, Ecole Polytechnique, Institut Polytechnique de Paris, Palaiseau, France \\
  	martin.krejca@polytechnique.edu
 }

\usepackage{graphicx}
\usepackage{mathtools}
\usepackage{amsmath}
\usepackage{amssymb}
\usepackage{enumitem}
\usepackage{amsthm}
\usepackage{hyperref}
\usepackage{svg}

\newcommand\numberthis{\addtocounter{equation}{1}\tag{\theequation}}

\begin{document}

\maketitle

\vspace*{-2\baselineskip}
\begin{abstract}
    Stochastic infection processes are continuous-time Markov chains on graphs that assign each vertex one of multiple states, such as \emph{susceptible}, \emph{infected}, or \emph{recovered}.
Depending on the model, vertices changes their state based on random transition rates and the states of their neighbors, resulting in a variety of complex dynamics.
The body of rigorous literature is rich for processes that consider a single infection, that is, processes where each state is unique.
In contrast, the setting with at least two infections, where the same state exists for different types, allows for far more transition combinations, leaving several interesting models entirely unexplored.
One such model is the one where vertices are immune to only one of the infections and lose this immunity upon infection by another one.

We address this shortcoming in the literature by defining the IRIR process, in which two SIR processes run on the same graph and each vertex is immune only to its most recent infection.
We study the \emph{survival time} of the \IRIR process, that is, the time until no infected vertex remains, with mathematical rigor.
Our main result is a tight threshold, known as \emph{epidemic threshold}, where the survival time rapidly changes from at most quasi-linear in the graph size~$n$ to at least super-polynomial in~$n$.
This result is applicable to perfectly mixed graphs, which are graphs where the density of edges between each non-empty subset of vertices is a given value $\density \in (0, 1]$.
Formally, denoting for infection type $i \in \{1, 2\}$ its infection rate by~$\lambda_i$ and its recovery rate by~$\varrho_i$, we show that the survival time is with high probability $2 \exp(\bigOmega{n / \ln(n)})$ if $\frac{\ra}{\la \density} + \frac{\rb}{\lb \density} < (1-c)n$, and that it is \bigO{n \ln(n)} if $\frac{\ra}{\la \density} + \frac{\rb}{\lb \density} > (1+c)n$, for a suitable constant $c \in \R_{> 0}$.
Our super-polynomial lower bound extends to jumbled graphs, which allow for some more flexibility in the density.
In particular, this includes with high probability \erdosRenyi graphs with an average degree of $k\in \smallOmega{\ln^2(n)}$.

Our proof for the lower bound is based on a \emph{potential} that transforms the configurations of the \IRIR process to a supermartingale with drift in a large region, implying the lower bound.
We detail how to systematically derive such a potential, based on a Lyapunov function of the transition equilibrium of the process.
We believe that this derivation is applicable to various other stochastic infection processes as well and thus constitutes a major part of our contribution.
\end{abstract}

\section{Introduction}
A large and important area in theoretical computer science considers stochastic processes on graphs driven by interactions among the different states of vertices.
Over the last decades, the analysis of such processes has led to various strong guarantees.
Examples include rumor spreading~\cite[e.g.,][]{KarpSSV00,DoerrFF11}, load balancing~\cite[e.g.,][]{FriedrichS09,CaiS17}, voter models~\cite[e.g.,][]{CooperEOR13,BerenbrinkGKMT16}, percolation~\cite[e.g.,][]{BaloghP07,FalgasRSarkar23}, and evolutionary dynamics~\cite[e.g.,][]{DiazGMRSS14,Galanis0GLR17}.
Stochastic infection processes~\cite{pastor2015epidemic} comprise another important domain that falls into this category and that has been studied with similar rigor.

Stochastic infection processes are motivated by epidemiology and assume that each vertex of a given graph is in one of several states, such as susceptible to an infection~(S), infected~(I), or recovered~(R) and thus immune to an infection.
The dynamics are governed by random rates that determine when vertices transition in continuous time from one state to another, based on the states of their neighbors.
Following a common naming convention, the order of possible transitions is typically reflected by the order of the abbreviation letters in the name of the process, with a repeated letter indicating a loop.
Well-known processes are the SIR process, which results in a type of percolation process, the SIS process, also known as contact process, and the SIRS process.
We note that processes with loops and thus repeated configurations lead to vastly different dynamics from processes without loops.
Thus, in the remainder, we refer to stochastic infection processes only if they allow for repeated configurations, as other models are also covered by other models, such as SIR and percolation.

Since stochastic infection processes usually have exactly one absorbing state, which is quickly reached once no infected vertices remain, an important quantity of a process is the random time that it takes to reach such an absorbing state.
This time is the \emph{survival time} of the process, and it has been subject to extensive research.
Many results work under idealized assumptions, known as mean-field theory, and study approximations of the exact processes.
Nonetheless, popular models have also been studied in full rigor.
For example, the SIS process has been analyzed both for infinite graphs~\cite[e.g.,][]{Harris74,Liggett96InfiniteBinaryTrees,NamNS22SISinfinite} and finite graphs~\cite[e.g.,][]{berger2005spread,ganesh2005effect,BorgsCGS10Antidote}.
More recently, the SIRS process has received more attention and was also studied in full rigor, on finite graphs,~\cite[e.g.,][]{FrGoKlKrPa2024,LamNY25,GoebelKKP25}.

Most of these results showcase an \emph{epidemic threshold}, that is, a value of the model's parameters---typically the infection rate---at which the survival time suddenly changes from at most logarithmic to at least super-polynomial, with high probability.
These epidemic thresholds are usually closely related to the structure of the host graph of the process, and the various results paint a rich picture of graph properties that result in a low or high survival time.
However, this richness of rigorous results only exists for processes that consider a single infection.
For two or more infections, the landscape of fully rigorous results is scarcer, as we detail in \Cref{sec:relatedWork}.
One substantial reason for this scarcity is the drastic increase in complexity when studying multiple infections, as there are many different choices with respect to what interactions are sensible.
In particular, the interesting scenario of multiple infections with mutually exclusive immunity has not been studied with mathematical rigor at all so far, to the best of our knowledge.
Such mutual exclusion shows up in certain contexts, for example, in the (helpful) Welchia virus that patches a system and thus changes its vulnerability, or as seen by crop rotation, where planting the same crop repeatedly increases the infection risk, making it likely to change the crop and allowing the soil to recover from the former crop but increasing the infection risk of the latter~\cite{ZhouEtAl23,IslamOLKMW24}.

\vspace*{-1 ex}
\paragraph{Our Contribution}
We introduce a new infection process for two competing infections with mutually exclusive immunity, called \emph{\IRIR}, and we locate with mathematical rigor its epidemic threshold on perfectly mixed graphs, which resemble weighted cliques.
To this end, we meticulously derive a suitable potential that transforms the complex configuration space into a supermartingale with drift, which we believe is a value contribution of this work and applicable to other processes as well.
We also provide estimates for the dynamics of a generalization of the \IRIR process to more than two infections.

\textbf{The \IRIR model.}
In the IRIR process, two SIR processes\footnote{In an SIR process, vertices transition from susceptible to an infection to infected to recovered and thus immune to the initial infection. Once recovered, a vertex remains in this state indefinitely.} are present on the same graph simultaneously.
We assume that each vertex can only be infected by one infection at a time and can also only be immune to one infection at a time, namely that one from which it recovered.
We note that when an immune vertex is infected (by the other infection), this overwrites the immunity of the vertex entirely.
Hence, the susceptible state of one infection is matched to the recovered state of the other infection. This property of overwriting the immunity resembles more an SIRS process\footnote{In an SIRS process, recovered vertices become susceptible again after a certain random time, allowing them to be reinfected.} than an SIR process, as the reinfections allow the infection to survive for a longer time.
Since it is impossible in our model to enter a state in which a vertex is susceptible to both infections, we drop this state for simplicity and assume that every vertex starts either infected or immune to one of the infections, hence the abbreviation \IRIR.

More detailed, the IRIR process is a continuous-time Markov chain on a graph that partitions the vertices into two infected and two recovered states---one of each state for either infection. Recovered vertices are immune to their own infection while being susceptible to the other one. For infection $i \in \{1,2\}$, infected vertices transition into the recovered state following an exponential clock with rate\footnote{We note that~$\varrho$ is traditionally used for R-to-S transitions in SIRS processes, and the recovery rate (from~I to~R) is normalized to~$1$. Since we have no R-to-S transitions in \IRIR but two recovery rates, we use~$\varrho$ for this transition instead.}~$\varrho_i$, and they infect adjacent vertices susceptible to~$i$, that is, recovered vertices of the other type, at rate~$\lambda_i$.

\textbf{The epidemic threshold.}
We analyze the \IRIR process on perfectly mixed graphs with $n \in \N_{\geq 1}$ vertices, where between any two disjoint subsets of vertices there are exactly as many edges as predicted given a density $\density \in (0,1]$. We show for constant recovery rates and infection rates in $\bigTheta{(n\density)^{-1}}$ that the survival time of an \IRIR process is with high probability super-polynomial if $\frac{\ra}{\la \density} + \frac{\rb}{\lb \density} < (1-c)n$ for some constant $c \in \R_{>0}$ (\Cref{thm:clique}). We complement this result with an upper bound of $\bigO{n \ln(n)}$ time if $\frac{\ra}{\la \density} + \frac{\rb}{\lb \density} > (1+c)n$ (\Cref{thm:fastDieOut}). Combined, these two results establish a threshold behavior that is very similar to the behavior of the SIS process, where the expected survival time changes from logarithmic to super-polynomial around $\frac{\varrho}{\lambda \density} = n$~\cite{ganesh2005effect}.

We extend the super-polynomial lower bound to jumbled graphs (\Cref{cor:jumbledContinuous}), which are a relaxation of perfectly mixed graphs. In particular, we consider $(p,p n/k)$-jumbled graphs for $k\in \smallOmega{\ln(n)}$, which bound the difference between the actual number of edges between vertex subsets and their expectation. With high probability, \erdosRenyi graphs with an average vertex degree of $\smallOmega{\ln^2(n)}$ are jumbled graphs in this sense, making our lower bound immediately applicable (\Cref{cor:ErdosSurvival}).

\textbf{Finding a suitable potential.}
Our proof of the super-polynomial lower bound is involved.
At its core lies a concentration result for supermartingales with drift (\Cref{pre:negativeDrift}), stemming from drift analysis~\cite{Lengler20BookChapter}.
To this end, we transform the trajectory of the \IRIR process---consisting of partitions of the vertices---into a real-valued random process via a \emph{potential} that allows to apply drift analysis.

We note that while this approach is not uncommon, the process of finding a suitable potential is challenging.
In this work, we discuss in detail how we systematically obtain such a suitable potential, which is, to the best of our knowledge, a novel and substantial contribution.
We believe that our strategy is applicable to other processes and thus serves as a good starting point for investigating other processes.

\textbf{Dynamics for more than two infections.}
Last, we analyze the dynamics of a generalization of the \IRIR process to more than two infections.
To this end, we calculate the transition equilibrium values for the larger system and argue which of the infections are likely to survive in such an equilibrium.

\vspace*{-1 ex}
\paragraph{Outline}
In \Cref{sec:relatedWork}, we discuss related work on stochastic infection processes with at least two infections.
We then define the \IRIR process fully formally in \Cref{sec:preliminaries} and present our notation as well as our mathematical tools.
Afterward, we present and discuss our main results in \Cref{sec:mainResult}.
In particular, we detail how we derive and apply our potential in \Cref{sec:model,sec:drift}.
In \Cref{sec:multipleInf}, we discuss our results for more than two infections.
Last, we provide an outlook in \Cref{sec:outlook}.

\section{Related Work}
\label{sec:relatedWork}
We discuss the novelty of our work with respect to existing articles.
Briefly put, the \IRIR process has not been studied before, and similar models were mostly only studied via mean-field, that is, approximation, approaches, not in full mathematical rigor.
Similarly, while our general proof strategy is not uncommon, we are not aware of precise instructions for deriving suitable potentials systematically.

\textbf{Novelty of the \IRIR process.}
To the best of our knowledge, the \IRIR process has not been studied before. However, its properties have similarities to processes studied in prior work. Hence, in the following, we outline several properties of our model and examples of works that studied similar properties. We start by presenting work on processes with multiple infections that compete or cooperate or do both. We then provide an example of a paper that also considers a process where vertices only memorize the last infection. Last, we relate our process to predator--prey systems, which have similar cyclic dynamics.

Processes with multiple infections received a lot of attention in general. \Textcite{kucharski2016capturing} give an overview on the research on multi-strain models in which multiple strains of the same infection spread through a network, where being infected or recovered from one strain gives (partial) immunity to other strains. The authors motivate the different models using real-world examples and discuss the advantages and disadvantages of certain modeling decisions. However, in terms of mathematical rigor, most papers in this area use mean-field approaches in order to calculate the equilibrium states and threshold values for when new equilibria arise.

The dynamics of the IRIR process result in infections that are cooperative in the sense of one infection infecting vertices being purely positive for the other infection, as the susceptible pools are disjoint and grow by recovering from the other infection. In that sense, there has been previous work on cooperative infections that help each other in different ways. For example, \textcite{chen2013outbreaks} consider two SIR-type infections that infect vertices on the same graph, and being infected by one infection increases the probability to be infected by the other one. The authors show phase transitions via mean-field theory.

Another work that studies cooperation is by \textcite{pinotti2020interplay}. The authors analyze a system of three infections where one of them cooperates with the other two, while the other two are in competition with each other. The focus is on the effect of having both cooperation and competition in the same model. The analysis uses mean-field calculations and simulations in order to calculate a phase diagram. The system dynamics are related to our extension of the \IRIR process to more than two infections.
The difference is that our \IRIR generalization results naturally in an interplay of cooperation and competition, whereas the model by \textcite{pinotti2020interplay} enforces each individually.

The \IRIR process can be seen as two SIR infections where vertices only remember the last infection and are thus immune to it. \textcite{guerin2025stochastic} study a similar infection model and analyze the impact of having memory of the last infection. However, in their model, each infected individual has a \emph{trait} parameter, which changes with each infection, and the age of the infection impacts its effectivity.

Because infections in the \IRIR process only infect vertices that recovered from the other infection, vertices cycle alternatingly through the two infected and recovered states. This is similar to cyclic predator--prey systems.
The main difference is that predator--prey models feature a high degree of asymmetry, with the predators hunting the prey, and both populations are subject to death-and-life-cycles.
Predator--prey models have been studied extensively, for example, by \textcite{he2012relationship} and \textcite{durrett1998spatial} as well as the references therein. \textcite{he2012relationship} showed that cyclic three-species predator--prey systems are closely related to the two-species Lotka--Volterra model, which has also been studied extensively. Most papers on both models describe the dynamics over differential equation systems and analyze the existence and stability of equilibria in the mean-field model. In this regard, the analysis of these processes is very close to how SI-type processes are commonly analyzed. Moreover, there are some fully rigorous results, mostly on underlying latices, which investigate the differences of the exact process to the mean-field approximation~\cite[e.g.,][]{tauber2012population}. In particular in the cyclic predator--prey model, the stochastic process can behave quite differently to the mean-field approximation, as for example discussed by \textcite{szolnoki2014cyclic}. The more complex the models are, the more impact local interactions can have on the global behavior of the dynamics. This highlights the importance of rigorous analysis of such processes that goes further than using mean-field assumptions.

\textbf{Novelty of our analysis.}
Using mean-field theory for the analysis of SIS-type infections is very common.
Similarly, Lyapunov functions are regularly used for showing global stability of equilibrium states. Among such works, closest to our analysis is the work by \textcite{FrGoKlKrPa2024}, who extended SIRS mean-field calculations by \textcite{korobeinikov2002lyapunov} to fully rigorous results for expander graphs. \textcite{FrGoKlKrPa2024} claim that their adjustments made to the potential should also work for similar models. For the \IRIR process, this is not the case as the increased number of parameters is not captured by the approach by \textcite{FrGoKlKrPa2024}. Hence, in this article, we provide in \Cref{sec:model} step-by-step instructions for how to build the Lyapunov function, explaining how the adjustments change the function's properties, why they are needed for the drift analysis, and how several parameters are incorporated.
Such detailed and systematic instructions are novel, to the best of our knowledge.

\section{Preliminaries}
\label{sec:preliminaries}
We first introduce general notation and terminology.
Then, we formally define the IRIR process, which is at the core of this article.
Afterward, we define the graph classes that we analyze, namely, \emph{perfectly mixed} and \emph{jumbled} graphs.
Last, we briefly discuss \emph{drift analysis}, which is the main mathematical tool we use.

\vspace*{-1 ex}
\paragraph{General notation.}
We let~$\N$ denote the set of all natural numbers, including~$0$, and we let~$\R$ denote the set of all reals.
For all $a, b \in \N$, let $[a .. b] \coloneqq [a, b] \cap \N$.

A graph is a finite, undirected graph without self-loops.
We reserve the variable $n \in \N_{\geq 1}$ for the number of vertices of a graph, and we consider all big-O notation to be asymptotic in this~$n$.
In particular, if we refer to a variable as a \emph{constant}, it is a value independent of~$n$.

A \emph{Poisson process of rate $\lambda \in \R_{> 0}$} is a one-dimensional Poisson point process of rate~$\lambda$ that outputs a random subset~$S$ of non-negative real numbers.
This subset~$S$ is almost surely countably infinite, and the difference between each two consecutive numbers in~$S$ follows an exponential distribution with parameter~$\lambda$, independent of all other random decisions.
We call Poisson point processes \emph{clocks} and say that the clock \emph{triggers} at the times in the returned subset~$S$.

\vspace*{-1 ex}
\paragraph{The IRIR process.}
The IRIR process is a continuous-time Markov chain over a graph that models two competing infections.
Each vertex is either infected by exactly one of the two infections or is recovered from exactly one of the two infections and thus immune to the infection from which it recovered.

Formally, given a graph $G = (V, E)$ as well as two \emph{infection rates} $\la, \lb \in \R_{>0}$, two \emph{recovery rates} $\ra, \rb \in \R_{>0}$,\footnote{~We note that, traditionally, the recovery rate is normalized to~$1$ in stochastic infection processes. However, since we have two recovery rates that may differ, we use explicit variables for both rates.} and an initial partition of~$V$ into the sets $I_{1,0}', I_{2,0}', R_{1,0}'$, and $R_{2,0}'$, an IRIR process $C=(C_t)_{t\in \R_{\geq 0}}$ is a continuous-time Markov chain over partitions of~$V$.

At each point in time $t \in \R_{\geq 0}$, the configuration~$C_t$ represents a partition of~$V$ into four sets:
the \emph{infected} sets $\Ia{t}', \Ib{t}'$ of infection~$1$ and~$2$, respectively, and the \emph{recovered} sets $\Ra{t}', \Rb{t}'$ of infection~$1$ and~$2$, respectively.
That is $C_t=(\Ia{t}',\Ib{t}',\Ra{t}',\Rb{t}')$.
Vertices that are recovered from one infection are susceptible to infection by the other and only the other infection.

Vertices change states following events triggered by clocks.
For each edge $e \in E$, we have two \emph{infection} clocks $M_{1,e}$ and $M_{2,e}$ with rates $\la$ and $\lb$, respectively.
For each vertex $v \in V$, we have two \emph{recovery} clocks $N_{1,e}$ and $N_{2,e}$ with rates $\ra$ and $\rb$, respectively.
Let~$Z$ be the union of all clocks, and let~$P$ be the stochastic process in which all clocks in~$Z$ run simultaneously and independently, starting at time~$0$.
Note that, almost surely, no two clocks trigger at the same point in time.
Last, we index the increasing sequence of all trigger times in~$P$ by $\{\gamma_i\}_{i\in\N}$, with $\gamma_0=0$.

Recall that configuration~$C_0$ is part of the input of~$C$.
The other configurations are defined inductively and only change when clocks trigger.
That is, for all $i \in \N$, the process~$C$ is constant on $[\gamma_i, \gamma_{i+1})$.
The transitions are defined for all $i \in \N$ and $s \in [\gamma_i, \gamma_{i + 1})$ as follows:
\begin{itemize}
      \item \textbf{Susceptible to infected.}
            Let $j \in \{1,2\}$, $e = \{u, v\} \in E$ with $\gamma_{i + 1} \in M_e$ and $u \in I_{j,s}'$, as well as $v \in R_{3-j,s}'$.
            Note that~$j$ and $3-j$ are the two different infections, hence,~$v$ is susceptible to infection~$j$.
            Then for all $t \in [\gamma_{i + 1}, \gamma_{i + 2})$, we have $u, v \in I_{j,t}'$.
            We say that \emph{$u$ infects~$v$} (at time~$\gamma_{i + 1}$).
      \item \textbf{Infected to recovered.}
            Let $j \in \{1,2\}$ as well as $v \in V$ with $\gamma_{i + 1} \in N_v$ and $v \in I_{j,s}'$.
            Then for all $t \in [\gamma_{i + 1}, \gamma_{i + 2})$, we have $v \in \R_{j,t}'$.
            We say that \emph{$v$ recovers} (at time~$\gamma_{i + 1}$).
\end{itemize}
The transitions above comprise all that occur.
However, some triggers do not result in a change of state, for example, an infection trigger between two recovered vertices.
Hence, in the remaining paper, we only consider times at which state changes occur.
That is, we consider $\{\gamma_0\} \cup \{\gamma_i \mid i \in \N_{\geq 1} \land \contactProcess_{\gamma_i} \neq \contactProcess_{\gamma_{i - 1}}\}$, which we index by the increasing sequence $\{\timeContinuous{i}\}_{i \in \N}$.
For all $i \in \N$, we call $\timeContinuous{i}$ the $i$-th \emph{step} of~$C$.

An important quantity in our analysis is the number of vertices in each state.
To this end, let $\Ia{t} \coloneqq |\Ia{t}'|$, $\Ib{t} \coloneqq |\Ib{t}'|$, $\Ra{t} \coloneqq |\Ra{t}'|$, and $\Rb{t} \coloneqq |\Rb{t}'|$.
Our analysis focuses on the \emph{survival time} $T\coloneqq \inf\{t \in \R_{\geq 0}\mid \Ia{t}= \Ib{t} = 0\}$ of $C$, and we say that the infection \emph{dies out} at~$T$.

\vspace*{-1 ex}
\paragraph{Perfectly mixed and jumbled graphs.}
Given a weighted graph $G = (V, E)$ and a \emph{density} $\density \in (0, 1]$, we say that~$G$ is \emph{perfectly mixed} if and only if for all non-empty disjoint subsets of vertices $U,W \subset V$, the sum of edge weights of edges between $U$ and $W$ is exactly $\density \cdot |U|\cdot |W|$.
Note that since this includes all singleton sets for~$U$ and~$W$, this implies that each pair of vertices is connected by an edge with weight~$\density$.
Hence, a perfectly mixed graph is a weighted clique, in which all edges have weight $\density$.

When running an IRIR process~$C$ on a weighted graph, we multiply the infection rates $\la, \lb$ by the weight of the edge to get the rate of the corresponding clock.
Note that on perfectly mixed graphs, all infection clocks have rates $\lae \coloneqq \la \density$ and $\lbe \coloneqq \lb \density$, respectively, which we call the \emph{effective infection rates}.
Furthermore, note that an infection on a perfectly mixed graph with density $\density$ and infection rates $\la$ and~$\lb$ is equivalent to an infection on a clique with infection rates $\lae$ and $\lbe$.

Moreover, we consider \emph{jumbled graphs}~\cite{thomason1987pseudo}, which are graphs that are close to being perfectly mixed.
Formally, let $G=(V,E)$ be a graph, $p \in (0,1)$, and $\alpha \in \R_{\geq 1}$. For a graph~$H$, let $e(H)$ denote its number of edges and~$|H|$ its number of vertices.
We call $G$ \emph{$(p,\alpha)$-jumbled} if and only if for each induced subgraph~$H$ of~$G$, we have $\bigl|e(H)-p\cdot\binom{|H|}{2}\bigr| \leq \alpha |H|$. In a sense, jumbled graphs are a de-randomized version of \erdosRenyi graphs $G(n,p)$, which are almost surely $(p',\bigO{\sqrt{np}})$-jumbled (\Cref{lem:ErdosJumbled}).

\vspace*{-1 ex}
\paragraph{Drift analysis.}
Drift analysis~\cite{Lengler20BookChapter} is a collection of mathematical tools for bounding the expected stopping times of real-valued randomized processes.
This terminology dates back to an article by \textcite{HeY03}, based on a paper by \textcite{Hajek82}.

At its core, drift analysis is a rephrasing of well-known results for super- and submartingales, such as the optional-stopping theorem.
The analysis is applicable to any real-valued process whose expected change is bounded in a certain way, with this bound usually being referred to as the \emph{drift} of the process.

In this article, we mainly make use of \emph{negative} drift.
That is, we consider a random process~$X$ that strictly decreases in expectation as long as it did not it reach a sufficiently large value, where we stop~$X$.
The following drift theorem, based on \textcite[Theoreom~$3$]{Koetzing16}, provides an upper bound on the probability of~$X$ being stopped before an exponential time passes.
In its essence, the theorem is an application of the Azuma--Hoeffding inequality to sub-Gaussian supermartingales.

\begin{restatable}[Negative Drift, {\protect\cite[Corollary~3.24]{krejca2019theoretical}}]{theorem}{negativeDrift}
      \label{pre:negativeDrift}
      Let $(X_t)_{t\in \N}$ be an integrable random process over~$\R$ adapted to a filtration $(\filtration_t)_{t \in \N}$.
      Furthermore, let $X_0\leq 0$, $b \in \R_{>0}$, and $T = \inf\{t \in \N \mid X_t \geq b\}$.
      Suppose that there are values $a \in \R_{\leq 0}$, $c \in (0,b)$, and $\varepsilon \in \R_{<0}$ such that for all $t \in \N$, we have
      \begin{enumerate}
            \item\label[cond]{item:negativeDrift:drift}
                  $\E{X_{t+1}-X_t}[\filtration_t] \cdot \indicator{X_t\geq a \land t < T} \leq \varepsilon \cdot \indicator{X_t\geq a \land t < T}$,
            \item\label[cond]{item:negativeDrift:boundedStepSize}
                  $|X_t-X_{t+1}|\cdot \indicator{X_t\geq a} < c \cdot \indicator{X_t\geq a \land t < T} + \indicator{X_t < a \lor t \geq T}$, and
            \item\label[cond]{item:negativeDrift:noLargeJumps}
                  $X_{t+1}\cdot \indicator{X_t< a \land t < T}\leq 0$.
      \end{enumerate}
      Then for all $t\in \N$, we have $\Pr{T\leq t} \leq t^2 \cdot \exp\bigl(-\frac{b|\varepsilon|}{2c^2}\bigr)$.
\end{restatable}

An important property of \Cref{pre:negativeDrift}, using the theorem's notation, is that we only require drift for the regime $[a, b)$ and not over the entire interval $(-\infty, b)$, as the bounded step size prevents the process from stopping earlier anyway once it takes values less than~$a$, due to \cref{item:negativeDrift:noLargeJumps}.

\section{Bounds on the Survival Time of the IRIR Process}
\label{sec:mainResult}
Our main result is \Cref{cor:jumbledContinuous}, which shows that the \IRIR process on a jumbled graph~$G$ of density $\density \in \R_{>0}$ survives with high probability for a time that is at least super-polynomial in the number $n \in \N_{\geq 1}$ of the vertices of~$G$ if the recovery rates are constant, the infection rates are in the order~$\frac{1}{\density n}$, and the process starts with a linear fraction of infected vertices of each infection.

\begin{restatable}{corollary}{jumbledContinuous}
    \label{cor:jumbledContinuous}
    Let $n \in \N_{\geq 1}$, $k \in \R_{>0}$ with $k \in \smallOmega{\ln(n)}$, $p \in (0,1]$ (possibly dependent on~$n$), and let~$C$ be an \IRIR process on a $(p,p n/k)$-jumbled graph~$G$ with~$n$ vertices.
    Let~$C$ have constant recovery rates $\ra,\rb \in \R_{>0}$ and effective infection rates $\lae=\ca/n$ and $\lbe=\cb/n$, for constants $\ca,\cb \in \R_{>0}$.
    Furthermore, let $c\in \R_{>0}$ be a constant such that $\frac{\ra}{\lae} + \frac{\rb}{\lbe} < (1-c)n$, and let $\varepsilon_h \in \R_{>0}$ be a sufficiently small constant such that $\Ia{0}\geq \varepsilon_h n$ and $\Ib{0}\geq \varepsilon_h n$.
    Last, let~$T$ be the survival time of~$C$.
    Then there exists a constant $c_T \in \R_{>0}$ such that for sufficiently large~$n$, we have $\Pr{T\leq \exp\bigl(\frac{c_T n}{8 \ln(n)}\bigr)}[][\big] \leq 2 \exp\bigl(-\frac{c_T n}{2\ln(n)}\bigr)$.
\end{restatable}

We note that an \erdosRenyi graph $G(n,p)$ is almost surely $(p',\bigO{\sqrt{np}})$-jumbled (\Cref{lem:ErdosJumbled}) when $p\leq 0.99$.
Hence, \Cref{cor:jumbledContinuous} carries over for $p=k/n\leq 0.99$ and $k \in \smallOmega{\ln^2(n)}$ (\Cref{cor:ErdosSurvival}).

Since jumbled graphs subsume cliques, \Cref{cor:jumbledContinuous} is very much in line with results for the SIS process~\cite{ganesh2005effect} on cliques (with one infection), where the process also exhibits with high probability a super-polynomial survival time once the infection rate is at least in the order of~$\frac{1}{n}$.
The main difference to the SIS constraint is that we require upper and lower bounds for the infection rates in \IRIR as both infections depend on each other and an imbalance among the infections may lead to one infection dying out fast which also results in the other one dying out fast.

We show that the constraint $\frac{\ra}{\lae} + \frac{\rb}{\lbe} < (1-c)n$ is tight up to constant factors by showing that the expected survival time of the \IRIR process is at most \bigO{n \ln(n)} if $\frac{\ra}{\lae} + \frac{\rb}{\lbe} > (1+c)n$.
Due to concentration inequalities such as Markov's inequality, a polynomial upper bound also holds with high probability.

\begin{restatable}{theorem}{fastDieOut}
    \label{thm:fastDieOut}
    Let $C$ be an \IRIR process on a perfectly mixed graph~$G$ with $n \in \N_{\geq 1}$ vertices and density $\density \in (0, 1]$.
    Let~$C$ have constant recovery rates $\ra,\rb \in \R_{>0}$ and effective infection rates $\lae=\ca/n$ and $\lbe=\cb/n$, for constants $\ca,\cb \in \R_{>0}$.
    Furthermore, let $c\in \R_{>0}$ be a constant such that $\frac{\ra}{\lae} + \frac{\rb}{\lbe} > (1+c)n$.
    Last, let~$T$ be the survival time of~$C$.
    Then $\E{T} \in \bigO{n \ln(n)}$.
\end{restatable}

In combination, \Cref{cor:jumbledContinuous} and \Cref{thm:fastDieOut} show that~$n$ acts as an epidemic threshold for $\frac{\ra}{\lae} + \frac{\rb}{\lbe}$, that is, a value where the survival time rapidly changes from at most polynomial to at least super-polynomial by just changing the effective infection rates by a constant factor.
This is comparable to~$n$ being the epidemic threshold for $\frac{\varrho}{\lambda^*}$ in the SIS process~\cite{ganesh2005effect}.

\vspace*{-1 ex}
\paragraph{Proof overview.}
Due to space constraints, our proofs are in the appendix.
Here, we outline the main idea for proving \Cref{cor:jumbledContinuous}, mentioning that \Cref{thm:fastDieOut} follows from similar but easier arguments.
In the following, we provide a high-level idea before we explain specific important steps in greater detail.

Our aim is to apply \Cref{pre:negativeDrift} to the \IRIR process~$C$.
This requires us to transform~$C$ into a real-valued process that has a sufficient drift (\cref{item:negativeDrift:drift}).
We call a function that performs such a transformation a \emph{potential function}.
Finding such a potential is the main challenge of the analysis.

Our final potential is based on a continuous Lyapunov function, which in turn incorporates information about the equilibrium of the \IRIR process.
This equilibrium is an idealized configuration where all transition probabilities are equal, which may not be attainable with discrete values.

In an idealized \IRIR process with continuous values, the equilibrium is a fixed point, meaning that the process cannot escape it.
For the actual, discrete, \IRIR process, the vicinity of this configuration still leads to configurations from which the \IRIR process cannot escape quickly with high probability.
This property is well modeled by negative drift.

Given that the equilibrium is based on a continuous view of the \IRIR process and that the discrete drift of a process can be seen as the difference quotient over a discrete step of the process, we first derive a Lyapunov function~$F$ whose derivative, that is, its continuous-time drift, is negative in a region around the equilibrium.
We then discretize~$F$ such that its important properties are still maintained within this region (\Cref{sec:model}).
Applying \Cref{pre:negativeDrift} to~$C$ with this potential (\Cref{sec:drift}) concludes the proof.

\subsection{Deriving a Suitable Potential}
\label{sec:model}
As explained above, we derive a Lyapunov function that we then discretize into our potential.

Let $G=(V,E)$ be a graph with $\numberOfVertices \in \N_{\geq 1}$ vertices and density $\density \in (0, 1]$, and let~$C$ be an \IRIR-process on~$G$.
Assume that~$C$ has recovery rates $\ra, \rb \in \R_{> 0}$ and effective infection rates $\lae, \lbe \in \R_{> 0}$.
By requiring all transition rates that change the configuration of~$C$ to be equal, we calculate the equilibrium state of~$C$.
The resulting equilibrium values are, using the notation from \Cref{sec:preliminaries},
\begin{align*}
    \Rae = \frac{\rb}{\lbe} \text{,\quad}
    \Rbe = \frac{\ra}{\lae}\text{,\quad}
    \Iae = (\numberOfVertices-\Rae - \Rbe)\frac{\rb}{\ra +\rb}\text{,\quad and\quad}
    \Ibe = (\numberOfVertices-\Rae - \Rbe)\frac{\ra}{\ra +\rb}.
\end{align*}
Note that these values are only between $0$ and $n$ if $\Rae + \Rbe <n$.
Solving for $\lae$ yields $\lae > (\numberOfVertices-\frac{\rb}{\lbe})^{-1}\ra$ as a threshold for a non-trivial equilibrium to arise.
We obtain an analogous condition for~$\lbe$.

In order to show that this equilibrium is globally stable, we define a Lyapunov function similar to \textcite{korobeinikov2002lyapunov} and \textcite{FrGoKlKrPa2024}.
To this end, let $\lyapunovHelper\colon \R^2 \to \R_{\geq 0}, (x, x^*) \mapsto x^* \left( \frac{x}{x^*} - \ln \frac{x}{x^*}-1\right)$.
We define for all $t \in \R_{\geq 0}$ and for the \emph{parameters} $\xa, \xb, \ya, \yb \in \R$ to be defined later
\begin{align*}
    \Lyapunov_t & = \Lyapunov(\Ia{t},\Ra{t},\Ib{t},\Rb{t}) = \xa \cdot \lyapunovHelper[\Iae,\Ia{t}]+ \ya \cdot \lyapunovHelper[\Rae,\Ra{t}]+ \xb \cdot \lyapunovHelper[\Ibe,\Ib{t}]+ \yb \cdot\lyapunovHelper[\Rbe,\Rb{t}].
\end{align*}
The Lyapunov function~$F$ contains one term for each of the four states that measures how close the number of vertices in that state is to its equilibrium.
In the following, we detail how to choose the parameters in order for the derivative of the Lyapunov function by time to be non-positive everywhere.

Let $t \in \R_{\geq 0}$.
By the definition of~$F_t$, its derivative with respect to~$t$ is~$0$ whenever $\Ra{t}=\Rae$ and $\Rb{t}=\Rbe$.
Consequently, requiring that the derivative with respect to~$t$ is non-positive everywhere, its derivative with respect to~$\Ra{t}$ and~$\Rb{t}$ is~$0$ at all states with $\Ra{t}=\Rae$ and $\Rb{t}=\Rbe$.
We use this property in order to derive two constraints for the parameters (\Cref{lem:derivativeR}).
These constraints help us fix their values later.

We obtain a third constraint by expressing the derivative with respect to $\da{t} \coloneqq \Ra{t} - \Rae$ and $\db{t} \coloneqq \Rb{t} - \Rbe$ (\Cref{lem:lastDegree}) and by choosing the parameters that simplify~$F_t$ and make its derivative with respect to~$t$ clearly non-positive (\Cref{lem:derivativeF}).

Given that we have three constraints and four parameters, we resolve the last degree of freedom in a way that makes all parameters constant for the regime of rates we consider in later part of the proof.

In order to derive a potential from~$F$, we require furthermore that configurations that are close to at least one infection dying out to be large, as \Cref{pre:negativeDrift} stops once the potential becomes too large.
Currently,~$F$ is large whenever \emph{any} of the states has few vertices in them, whereas having only few recovered vertices does not imply that any infection dies out quickly.
Hence, we adjust the terms for~$\Ra{t}$ and~$\Rb{t}$ in~$F$ such that they never become too large.
This results for all $t \in \R_{\geq 0}$ in
\begin{align*}
    \Lyapunovb_t & = \xa\lyapunovHelper[\Iae,\Ia{t}]+ 2\ya\lyapunovHelper[2\Rae,\Ra{t}+\Rae]+ \xb\lyapunovHelper[\Ibe,\Ib{t}]+ 2\yb\lyapunovHelper[2\Rbe,\Rb{t}+\Rbe].
\end{align*}
whose derivative is still non-positive everywhere (\Cref{lem:derivativeK}).

Last, we require that the derivative is not only non-positive but also smaller than a negative constant (due to \cref{item:negativeDrift:drift}) for a sufficiently large region of the range.
The function~$\Lyapunovb$ does not satisfy this, as it is~$0$ whenever $\Ra{t}=\Rae$ and $\Rb{t}=\Rbe$.
However, its derivative is already sufficiently small whenever~$\da{t}$ or~$\db{t}$ are at least linear in~$n$, which is a satisfactory regime.
Hence, we adjust~$\Lyapunovb$ slightly in order to make its derivative smaller for small~$\da{t}$ and~$\db{t}$ without changing it greatly for the other values (\Cref{def:lyapunovPotential3}).
We call this final Lyapunov function~$\Lyapunovc$ (indexed by the time $t \in \R_{\geq 0}$) and use it as a potential.

\subsection{Applying the Negative Drift Theorem}
\label{sec:drift}
We apply \Cref{pre:negativeDrift} to the potential~$\Lyapunovc$ derived in \Cref{sec:model}.
We first derive a result for perfectly mixed graphs and then extend the proofs to jumbled-graphs, requiring only minor adjustments.

Recall that \Cref{pre:negativeDrift} only requires us to have negative drift for a sufficiently large region of the potential's range.
To this end, we relate the state space and the potential's range, deriving useful properties on the states that fall into the regime that are of use for further calculations.
Specifically, the potential is only large when~$\Ia{t}$ or~$\Ib{t}$ is small.
We show that there are constants $\varepsilon_l, \varepsilon_h, c_l, c_h \in \R_{>0}$ such that $\Lyapunovc_t < c_h n$ implies $\Ia{t} > \varepsilon_l n$ and $\Ib{t} > \varepsilon_l n$  (\Cref{lem:potentialBothBig}), and additionally $\Lyapunovc_t > c_l n$ implies $\Ia{t} < \varepsilon_h n$ or $\Ib{t} < \varepsilon_h n$ (\Cref{lem:potentialOneSmall}, respectively).
Hence, as long as the potential is less than a specific constant fraction of~$n$, neither~$\Ia{t}$ nor~$\Ib{t}$ is too small. For the other properties we assume that this condition is satisfied.

\Cref{item:negativeDrift:boundedStepSize,item:negativeDrift:noLargeJumps} of \Cref{pre:negativeDrift} are easily satisfied by showing that changing the number of infected or recovered vertices changes the potential, up to additive lower-order terms, by the derivative of~$\Lyapunovc$ with respect to the that state (\Cref{lem:differenceDerivative}), which is bounded from above by a constant (\Cref{lem:constantStep}).

For \cref{item:negativeDrift:drift} of \Cref{pre:negativeDrift}, we first show that the drift is the derivative of~$\Lyapunovc_t$ with respect to~$t$ normalized by the total rate of change~$r_t$ plus lower-order terms (\Cref{lem:driftDerivative}).
As $r_t \in \bigTheta{n}$ in the considered states, the drift depends mostly on the derivative of~$\Lyapunovc_t$, which is negative by construction.
We then show that this derivative in the considered regime is at most of order $-n/\ln(n)$ (\Cref{lem:drift}).
We show this via a case distinction on whether~$\Ra{t}$ or~$\Rb{t}$ is close to its equilibrium.
If one of these values is far away, then the drift of~$\Lyapunovb_t$ is already in $-\bigOmega{n}$, which is almost unaffected by the shift to~$\Lyapunovc_t$.
If both~$\Ra{t}$ and~$\Rb{t}$ are close, the drift of $\Lyapunovb_t$ is small, but the shift to~$\Lyapunovc$ yields a derivative in $-\bigOmega{n/\ln(n)}$.

Combining all these results and thus applying \Cref{pre:negativeDrift}, we obtain a bound on the distribution of the number of (discrete) steps until the infection dies out for perfectly mixed graphs (\Cref{thm:clique}).
We extend this result to jumbled graphs by showing  that the drift on jumbled graphs within a given parameter range is the same as on perfectly mixed graphs up to lower order terms (\Cref{lem:cliqueJumbled}).
All other properties we showed carry over directly.
Hence, the bounds for jumbled graphs and perfectly mixed graphs are the same.
Last, we transfer these results to continuous time by bounding with high probability how much continuous time passes at most between two discrete steps (\Cref{cor:jumbledContinuous}).

\section{Dynamics for More than Two Infections}
\label{sec:multipleInf}
We generalize the \IRIR process to more than two infections and derive the transition equilibria.
With these insights, we speculate on the resulting dynamics, using our insights from \Cref{sec:mainResult}.

The \IRIR process naturally generalizes to more than two infections.
Given $k \in \N_{\geq 3}$ infections, each infection $i \in [1 .. k]$ has its own effective infection rate $\lambda^*_i \in \R_{> 0}$ and its own recovery rate $\varrho_i \in \R_{> 0}$.
If a vertex infected by infection~$i$ recovers, it is susceptible to infection by all infections except infection~$i$.
Moreover, each vertex is either infected by or recovered from exactly one infection at a time.
Using the same notation as for \IRIR introduced in \Cref{sec:preliminaries} but extending the range of indices for all~$k$ infections, we obtain for each $i \in [1..k]$, abbreviating $[k]_i \coloneqq [1 .. k] \smallsetminus \{i\}$, that
\begin{align}
    \frac{\d I_{i,t}}{\d t} & = \lambda_i^* I_{i,t}\left(\sum\nolimits_{j \in [k]_i}{(R_{j,t})} - \frac{\varrho_i}{\lambda_i^*}\right)\label{der:Ii} \\
    \frac{\d R_{i,t}}{\d t} & = \varrho_i I_{i,t} - R_{i,t} \sum\nolimits_{j\in [k]_i}{(\lambda_j^* I_{j,t})} \label{der:Ri}.
\end{align}

The equilibrium states are exactly those where all derivatives of the equations above are~$0$.
This results in some trivial equilibria.
For example, for each $i \in [1 .. k]$ and $t \in \R_{\geq 0}$, we have $\frac{d I_{i,t}}{dt} = 0$ if $I_{i,t}=0$.
If this is the case for some of the $I_{i,t}$, the problem reduces to finding equilibria in a process with one fewer infection.
Hence, for the remainder, we assume that there is a $t \in \R_{\geq 0}$ such that for all $i \in [1 .. k]$, we have $I_{i,t} > 0$.
Let $i \in [1 .. k]$.
Then \cref{der:Ii} implies $\sum_{j\in [k]_i}{(R_{j,t})} = \frac{\varrho_i}{\lambda_i^*}$.
Summing this equation for all $j \in [k]_i$ and subtracting the equation for~$i$ multiplied by $k-2$ and then dividing by $k-1$ yields
\begin{align}
    R_i^*=\frac{\sum_{j\in [k]_i}{(\varrho_j/\lambda_j^*)-(k-2)(\varrho_i/\lambda_i^*)}}{k-1}.\label{eq:Ri}
\end{align}

In order for infection~$i$ to not die out quickly,~$R_i^*$ needs to be positive as otherwise $I_i^*$ is non-positive. Hence, we require
\begin{align}
    \frac{\varrho_i}{\lambda_i^*} < \frac{1}{k-2}\sum\nolimits_{j\in [k]_i}{\frac{\varrho_j}{\lambda_j^*}}.\label{ineq:i}
\end{align}

In a sense, $\frac{\varrho_i}{\lambda_i^*}$ can be seen as the \emph{weakness score} of infection $i$. It says how many vertices susceptible to infection $i$ are required in order to sustain the number of vertices infected by~$i$ in expectation. If the weakness score is too high compared to the weakness scores of the other infections, infection~$i$ is not able to survive against the other infections. \Cref{ineq:i} shows that an infection can only be slightly weaker than the average infection in the process in order to survive, in particular if~$k$ is large.

An interesting observation is that when all but one infection have the same weakness score, they all satisfy \cref{ineq:i} as long as the last infection has a positive recovery rate. Hence, a single strong infection cannot make the other infections die out. Instead, it creates a more competitive environment among the other infections where an infection dies out if it is only slightly weaker than the others.

We get another restriction for survival by summing \cref{eq:Ri} for all $i \in [1 .. k]$, resulting in
\begin{align}
    R^* \coloneq \sum\nolimits_{i \in [1..k]}{R_i^*}=\frac{1}{k-1} \sum\nolimits_{i \in [1..k]}{\frac{\varrho_i}{\lambda_i^*}}\label{eq:sum_Ri}.
\end{align}

As there are only~$n$ vertices,~$R^*$ needs to be smaller than~$n$ in order for the equilibrium to possibly be stable. Thus, the overall weighted weakness score (the right-hand side of \cref{eq:sum_Ri}) of the infections cannot be too large. Noting that \cref{ineq:i} is equivalent to $\frac{1}{k-1} \sum_{i \in [1..k]}{\frac{\varrho_i}{\lambda_i^*}} < \frac{1}{k-2}\sum_{j\in [k]_i}{\frac{\varrho_j}{\lambda_j^*}}$, we see that an infection $i \in [1 .. k]$ has a negative equilibrium value of~$R_i^*$ if and only if its addition increases $R^*$. Consequently, under idealized assumptions, infections die out until the remaining set~$S$ of infections minimizes $\frac{1}{|S|-1} \sum_{i \in S}{\frac{\varrho_i}{\lambda_i^*}}$. Note that a single infection alone cannot survive long, hence $|S| \geq 2$. If $\frac{1}{|S|-1} \sum_{i \in S}{\frac{\varrho_i}{\lambda_i^*}} \geq n$, there is no set of infections that is able to survive long and all infections die out. Otherwise, there is a stable state in which all values are positive.

We shift our attention from the equilibrium values of the~$R$ values to those of the~$I$ values.
Recall that at the equilibrium, all derivatives from \cref{der:Ri} are~$0$. Rearranging them yields for all $i\in [1..k]$ that
\begin{align}
    \frac{\varrho_i I_{i}^*}{R_{i}^*}+\lambda_i^*I_i^*=\sum\nolimits_{j\in [1..k]}{\lambda_j^* I_j^*}.\label{eq:Ii_inter}
\end{align}

Note that the term on the right of \cref{eq:Ii_inter} is independent of~$i$. Hence, the left-hand term is the same for all $i \in [1 .. k]$. Recalling that $\frac{\varrho_i}{\lambda_i^*}= \sum_{j\in [k]_i}{R_j^*}$, the left-hand term of \cref{eq:Ii_inter} is equivalent to $\frac{\lambda_i^*}{R_i^*}I_i^*(\frac{\varrho_i}{\lambda_i^*}+R_i^*)=\frac{\lambda_i^*}{R_i^*}I_i^*R^*$. Thus,~$I_i^*$ is proportional to~$\frac{R_i^*}{\lambda_i^*}$. Using $\sum_{i\in[1..k]}{I_i^*}=n-R^*$, we obtain
\begin{align}
    I_i^*=\frac{R_i^*/\lambda_i^*}{\sum_{j\in[1..k]}R_j^*/\lambda_j^*}(n-R^*).\label{eq:Ii}
\end{align}

All infections combined have $n-R^*$ infected vertices in the equilibrium. Hence, they have more infected vertices the smaller the overall weighted weakness of the infections is. From this pool of infected vertices, the stronger infections get more vertices. However, infections have fewer infected vertices in the equilibrium when the infection rate is larger and the weakness score is the same. In other words, both an increased infection and recovery rate makes the infection have fewer infected vertices in an equilibrium. Roughly speaking, this is a consequence of the weakness score deciding how many vertices become infected by an infection. Increasing both the infection and recovery rate by the same factor does not change the weakness score, but it makes vertices stay infected for a shorter time, meaning that at the same rate of infecting vertices, fewer vertices are infected at any given time.

All in all, the \IRIR dynamics with three of more infections differ from those with two infections substantially to some degree.
For two infections, each vertex is susceptible to at most one of the infections at a time.
That is, the dynamics are fully cooperative.
For $k \in \N_{\geq 3}$ infections, a recovered vertex is susceptible to $k - 1 \geq 2$ infections, which creates a competitive environment among the infections.
This competition is so strong that weak infections die out quickly while the remaining ones continue to thrive.

\section{Outlook}
\label{sec:outlook}
While we discuss in \Cref{sec:multipleInf} the dynamics of an \IRIR process with more than two infections under idealized assumptions, we are currently missing a fully rigorous analysis in the same spirit as our results in \Cref{sec:mainResult}.
In particular, it would be interesting to see whether our predictions translate to the fully rigorous setting.
Moreover, while our discussion for more than two infections predicts how the strongest equilibrium looks like, it does not give estimates for how likely it is.
These probabilities highly depend on the starting configuration and detailed infection interactions, which requires a more careful analysis.

Another interesting direction for future work is the analysis of more complex, heterogeneous graphs.
Currently, the impact of high-degree vertices on the survival time is unknown.
Particularly in the case with more than two infections, the structure of the graph could influence the dynamics drastically.

Another venue to explore is to relax our assumptions.
Our main results determine an epidemic threshold when all equilibrium values are of linear size.
It remains an open problem what happens when this is not the case because one of the infections is much stronger than the other one or because both infections are so strong that there are almost no recovered vertices.

Last, it is interesting to derive a stronger upper bound than $\bigO{n \ln(n)}$ on the survival time for perfectly mixed graphs if $\frac{\ra}{\lae} + \frac{\rb}{\lbe} > (1+c)n$.
Most SIRS-type infection processes exhibit an expected survival time of $\bigO{\ln(n)}$ in a similar regime.
However, in the \IRIR setting, it may be the case that $\bigO{\ln(n)}$ cannot be achieved because $\frac{\ra}{\lae} + \frac{\rb}{\lbe} > (1+c)n$ only implies that \emph{one} of the infections dies out fast. The other infection may then continue, similarly to a percolation process, for longer than $\bigO{\ln(n)}$ in expectation.

\printbibliography

\newpage
\appendix

\section{Mathematical Tools}
\label{sec:tools}
In this section, we state some mathematical tools we use in our analysis. We start by stating a property of $(p,\alpha)$-jumbled graphs.

\begin{lemma}
    Let $G=(V,E)$ be a graph and let $p \in (0,1)$ and $\alpha \in \R_{\geq 1}$. If $G$ is $(p,\alpha)$-jumbled, then for every two disjoint subsets $W,U \subset V$ holds
    \begin{align*}
        \left|e(W,U)-p|U||W|\right| & \leq 2 \alpha (|U|+|W|).\qedhere
    \end{align*}
\end{lemma}

This follows directly from expressing the number of edges between $W$ and $U$ as the number of edges in the graph induced by $W \cup U$ minus the edges in the graphs induced by $W$ and $U$ and summing up the error terms.

The next lemma shows that \erdosRenyi graphs are almost surely $(p,\sqrt{np})$-jumbled. It follows directly from \cite[Corollary~2.3]{krivelevich2006pseudo}.

\begin{lemma}[{\protect\cite[Corollary~2.3]{krivelevich2006pseudo}}]\label{lem:ErdosJumbled}
    Let $p = p(n)\leq 0.99$. Let $G=G(n,p)$ and let $p'$ be its density. Then $G$ is almost surely $(p',\bigO{\sqrt{np}})$-jumbled.
\end{lemma}

To bound the survival time of the infection process we use the following drift theorems.

\begin{theorem}[Additive Drift, {\protect\cite[Theorem~3.3]{krejca2019theoretical}}]\label{pre:additiveDrift}
    Let $(X_t)_{t \in \N}$ be random variables over $\R$ adapted to a filtration $(\filtration_t)_{t \in \N}$, and let $T= \inf\{t\in\N \mid X_t \leq 0\}$. Furthermore, suppose that,

    \begin{enumerate}
        \item there is some value $\delta>0$ such that, for all $t\in \N$, it holds that
              $$\E{(X_t-X_{t+1})\cdot \indicator{t<t}\mid \filtration_t}\geq \delta \cdot \indicator{t<T} \text{, and that,}$$
        \item for all $t\in \N$, it holds that $X_t \cdot \indicator{t\leq T}\geq 0$.
    \end{enumerate}

    Then
    \begin{align*}
        \E{T} & \leq \frac{\E{X_0}}{\delta}.\qedhere
    \end{align*}
\end{theorem}

\negativeDrift*

To bound the drift, we use following lemma from \cite{FrGoKlKrPa2024} which gives useful bounds for the function $\lyapunovHelper$ defined in \Cref{def:laypunocHelper}.

\begin{lemma}[{\protect\cite[Lemma~4.3]{FrGoKlKrPa2024}}]\label{lem:lyapunovHelper}
    Let $x^* \in \R_{>0}$ and $x \in \R_{>2}$. Then
    \begin{align*}
         & \lyapunovHelper[x^*,x+1] - \lyapunovHelper[x^*,x] \leq 1 - \frac{x^*}{x} + \frac{x^*}{x(x+1)} \textrm{ and}          \\
         & \lyapunovHelper[x^*,x-1] - \lyapunovHelper[x^*,x] \leq -\left(1 - \frac{x^*}{x} - \frac{x^*}{x(x-1)}\right).\qedhere
    \end{align*}
\end{lemma}

We use the following theorem to calculate the expectation of the sum of a random number of random variables.

\begin{theorem}[Wald's Equation {\protect\cite[Chapter 10.2, Lemma 9]{grimmett2020probability}}]\label{lem:Wald}
    Let $(X_t)_{t\in \N_{>0}}$ be a sequence of independent, identically distributed random variables over $\R$, and let $T$ be a stopping time with respect to the natural filtration of $X$. Then
    \begin{align*}
        \E{\sum_{t=1}^{T}{X_t}} & = \E{X_1}\cdot \E{T}.\qedhere
    \end{align*}
\end{theorem}

We use the following concentration bound to bound the sum of exponential random variables such as the lengths of the steps of the process.

\begin{lemma}[{\protect\cite[Theorem~5.1]{janson2018tail}}]\label{lem:sumExponential}
    Let $X= \sum_{i=1}^{n}{X_i}$ with $X_i \sim \text{Exp}(a_i)$ independent. Further let $\mu \coloneqq \E{X} = \sum_{i=1}^{n}{\frac{1}{a_i}}$ and $a_* \coloneqq \min_i{a_i}$. Then for any $\lambda \in (0,1)$
    \begin{align*}
        \Pr{X \leq \lambda \mu} & \leq e^{-a_*\mu(\lambda - 1 - \ln(\lambda))}.\qedhere
    \end{align*}
\end{lemma}

\section{Deriving the Lyapunov function}\label{ap:model}

We aim to analyze the \IRIR process and show that it survives super-polynomially long on jumbled graphs in expectation. We start by analyzing some basic properties of the process assuming the graph is perfectly mixed, meaning all subsets of vertices are connected by as many edges as would be expected with a given density. Note that cliques are the only unweighted graphs which are actually perfectly mixed as the property enforces each pair of vertices to be connected with $\density$ edges. But jumbled graphs get very close to it and we show later on that the difference to the assumed perfectly mixed graph is not too big. Let $G=(V,E)$ be a graph with $\numberOfVertices$ vertices and density $\density$ and let $C$ be a \IRIR-process on $G$. For the analysis of the survival time we will look at the discrete version of $C$, but in order to get a feeling of the process, we look at the continuous version first. Assuming $G$ to be perfectly mixed, two disjoint subsets of vertices $X$ and $Y$ have exactly $\density |X|\cdot |Y|$ edges between them. As the infection rate is only relevant for infection over edges, it appears in the formulas always multiplied by the density. To make the terms easier we use the effective infection rates $\lae \coloneqq \la \density$ and $\lbe \coloneqq \lb \density$. Recall that \Ia{t}, \Ra{t}, \Ib{t}, and \Rb{t} denote the number of vertices in the respective state at time $t$. For perfectly mixed $G$, we get the following rates at which those values change, which are just the derivatives by $t$:
\begin{align}
    \frac{d \Ia{t}}{dt}= \lae \Ia{t} \Rb{t} - \ra \Ia{t} & \text{\quad \quad and \quad \quad} \frac{d \Ib{t}}{dt}= \lbe \Ib{t} \Ra{t} - \rb \Ib{t}\label{der:I} \\
    \frac{d\Ra{t}}{dt}= \ra \Ia{t} - \lbe \Ib{t} \Ra{t} & \text{\quad \quad and \quad \quad}\frac{d \Rb{t}}{dt}= \rb \Ib{t} - \lae \Ia{t} \Rb{t}.\label{der:R}
\end{align}

We are interested in the equilibrium values of that system. That are the values for \Ia{t}, \Ra{t}, \Ib{t}, and \Rb{t} such that all of the derivatives are 0. At those values the system tends to be relatively stable and we will also show that in most other states the system tends to drift in some sense towards the equilibrium (unless one of the infections died out). We denote the equilibrium values by \Iae, \Rae, \Ibe, and \Rbe. Setting both equations in \Cref{der:I} to 0 immediately gives us that $\Rbe = \frac{\ra}{\lae}$ and $\Rae = \frac{\rb}{\lbe}$. We then get the values for \Iae and \Ibe with \Cref{der:R}, by plugging in $\Rae$ and using that $\Iae+ \Rae+ \Ibe+ \Rbe=\numberOfVertices$. The resulting equilibrium values are:
\begin{align}
    \Rae = \frac{\rb}{\lbe} \text{;\quad}
    \Rbe = \frac{\ra}{\lae}\text{;\quad}
    \Iae = (\numberOfVertices-\Rae - \Rbe)\frac{\rb}{\ra +\rb}\text{;\quad}
    \Ibe = (\numberOfVertices-\Rae - \Rbe)\frac{\ra}{\ra +\rb}\label{eq:equi}.
\end{align}

During the process it is very important how close $\Ra{t}$ and $\Rb{t}$ are to their equilibrium values as this mainly decides the dynamics of the system. Therefore, we define $\da{t} \coloneqq \Ra{t} - \Rae$ and $\db{t} \coloneqq \Rb{t} - \Rbe$.

Note that this process is very symmetric in the sense that swapping all 1's in the index of a true equation with 2's and vice versa yields another true equation. We use this a lot in the coming analysis to only calculate one case and get the symmetric case for free.

Also note that these equilibrium points are only between 0 and $\numberOfVertices$ if $\Rae + \Rbe <n$. This inequality gives us an estimation of the epidemic threshold as it tells us that there exists an equilibrium other than the disease-free equilibrium (all states in which $\Ia{t}=\Ib{t}=0$ are equilibria as they cannot change anymore) if and only if it is fulfilled. Solving it for $\lae$ gives us the condition $\lae > (\numberOfVertices-\frac{\rb}{\lbe})^{-1}\ra$ (and a symmetric inequality for $\lbe$).

\subsection{Global stability of the equilibrium}

We now show that this equilibrium is globally stable if $\Rae + \Rbe <n$. To this end we define a Lyapunov function similar to the one in \cite{korobeinikov2002lyapunov} and \cite{FrGoKlKrPa2024} and show that the equilibrium is its global minimum and that its derivative is non-positive for all states (except for when one variable is zero as it is not defined then). We first define an auxiliary function \lyapunovHelper.

\begin{definition}\label{def:laypunocHelper}
    Let $\lyapunovHelper\colon (\R_{> 0})^2 \to \R$ be such that, for all $x,x^* \in \R_{>0}$, we have
    \begin{align*}
        \lyapunovHelper[x^*,x] & = x^* \left( \frac{x}{x^*} - \ln \frac{x}{x^*}-1\right).\qedhere
    \end{align*}
\end{definition}

Note that the derivative $\frac{\mathrm{d} f(x^*,x)}{\mathrm{d}x}=1- \frac{x^*}{x}$. Hence, for a given $x^* \in \R_{>0}$, the value $x = x^*$ is the only local optimum of $\lyapunovHelper[x^*,x]$, and it is a global minimum. The $x^*$ therefore acts like a target value and $\lyapunovHelper[x^*,x]$ gets larger the further $x$ is away from that target. It is common to use this function for infection processes and plug in the values \Ia{t}, \Ra{t}, \Ib{t}, and \Rb{t} and their respective equilibrium values. This gives us the following Lyapunov function:

\begin{definition}\label{def:lyapunovPotential}
    Let $G$ be a graph with $\numberOfVertices\in \N$ vertices and density $\density \in (0,1]$ and let \contactProcess be a \IRIR process on $G$ with $\Rae + \Rbe <n$.
    For all $t \in \R$, we define $\Lyapunov_t$ as
    \begin{align*}
        \Lyapunov_t & = \Lyapunov(\Ia{t},\Ra{t},\Ib{t},\Rb{t}) = \xa\lyapunovHelper[\Iae,\Ia{t}]+ \ya\lyapunovHelper[\Rae,\Ra{t}]+ \xb\lyapunovHelper[\Ibe,\Ib{t}]+ \yb\lyapunovHelper[\Rbe,\Rb{t}].\qedhere
    \end{align*}
\end{definition}

Note that outside of this section, we set the parameters as described in \Cref{eq:parameters}.

In this section, we aim to choose the parameters \xa, \ya, \xb, and \yb such that the derivative of the Lyapunov function is non-positive for all states. To give an insight on how the parameters are chosen, we first derive some conditions that these parameters have to fulfill in order to yield us a non-positive derivative everywhere.

Note that the derivative of $\Lyapunov_t$ by $t$ becomes 0 when $\Ra{t}=\Rae$ and $\Rb{t}=\Rbe$. The reason for that is that the equilibrium values are chosen such that they make the derivatives of the $I$'s 0. By definition of $\lyapunovHelper$, its derivative by the $R$'s is also 0 at their respective equilibrium.

We aim to find parameters such that this derivative is non-positive for all states. As we know that it is zero whenever $\Ra{t}=\Rae$ and $\Rb{t}=\Rbe$, we know that the derivative of this term with respect to $\Ra{t}$ has to be zero at those points. However, just taking the derivative by $\Ra{t}$ does not make much sense as only increasing or decreasing $\Ra{t}$ without changing the other terms moves us out of the valid state space as the four random variables have to add to $\numberOfVertices$. Hence, we take the derivative with respect to $\Ra{t}-\Ia{t}$ instead. We get the same effect by just replacing $\Ia{t}$ by $\numberOfVertices - \Ra{t}- \Rb{t}- \Ib{t}$ before taking the derivative by $\Ra{t}$. We get the following lemma.

\begin{restatable}{lemma}{derivativeR}\label{lem:derivativeR}
    Let $G$ be a perfectly mixed graph with $\numberOfVertices \in \N$ vertices and density $\density\in (0,1]$ and let \contactProcess be a \IRIR process on $G$ with $\Rae + \Rbe <n$. Consider all states with $\Ra{t}=\Rae$ and $\Rb{t}=\Rbe$ and consider the function $\Lyapunov_t$ from \Cref{def:lyapunovPotential} with the substitution $\Ia{t} = n - \Ib{t} - \Ra{t} - \Rb{t}$. Then we get
    \begin{align*}
        \left(\forall{\Ia{t}\in (0,n-\Rae-\Rbe)}:\frac{d^2\Lyapunov_t}{dt \cdot d \Ra{t}}=0\right) & \Leftrightarrow \xb =\frac{\ra+\rb}{\rb}\ya.\qedhere
    \end{align*}
\end{restatable}

\begin{proof}
    We start by calculating the derivative of the Lyapunov function.

    \begin{align*}
        \frac{d\Lyapunov_t}{dt}= & \xa \frac{d\lyapunovHelper[\Iae,\Ia{t}]}{dt}+\ya \frac{d\lyapunovHelper[\Rae,\Ra{t}]}{dt}+\xb \frac{d\lyapunovHelper[\Ibe,\Ib{t}]}{dt}+\yb \frac{d\lyapunovHelper[\Rbe,\Rb{t}]}{dt} \\
        =           & \xa \frac{d\lyapunovHelper[\Iae,\Ia{t}]}{d\Ia{t}}\frac{d\Ia{t}}{dt}+\ya \frac{d\lyapunovHelper[\Rae,\Ra{t}]}{d\Ra{t}}\frac{d\Ra{t}}{dt}                                           \\
                    & +\xb \frac{d\lyapunovHelper[\Ibe,\Ib{t}]}{d\Ib{t}}\frac{d\Ib{t}}{dt}+\yb \frac{d\lyapunovHelper[\Rbe,\Rb{t}]}{d\Rb{t}}\frac{d\Rb{t}}{dt}                                           \\
        =           & \xa(1-\frac{\Iae}{\Ia{t}})(\lae \Ia{t} \Rb{t}- \ra \Ia{t})                                                                                                                      \\
                    & +\ya(1-\frac{\Rae}{\Ra{t}})(\ra \Ia{t}-\lbe \Ib{t}\Ra{t})                                                                                                                       \\
                    & +\xb(1-\frac{\Ibe}{\Ib{t}})(\lbe \Ib{t} \Ra{t}- \rb \Ib{t})                                                                                                                     \\
                    & +\yb(1-\frac{\Rbe}{\Rb{t}})(\rb \Ib{t}-\lae \Ia{t}\Rb{t})                                                                                                                       \\
        =           & \xa \lae \Ia{t} \Rb{t} - \xa \ra \Ia{t} - \xa \lae \Iae \Rb{t} + \xa \ra \Iae                                                                                                   \\
                    & +\ya \ra \Ia{t} -\ya \lbe \Ib{t} \Ra{t} - \ya \frac{\ra \rb}{\lbe} \frac{\Ia{t}}{\Ra{t}}+\ya \rb \Ib{t}                                                                         \\
                    & +\xb \lbe \Ib{t} \Ra{t} - \xb \rb \Ib{t} - \xb \lbe \Ibe \Ra{t} + \xb \rb \Ibe                                                                                                  \\
                    & +\yb \rb \Ib{t} -\yb \lae \Ia{t} \Rb{t} - \yb \frac{\ra \rb}{\lae} \frac{\Ib{t}}{\Rb{t}}+\yb \ra \Ia{t}                                                                         \\
        =           & \xa \Iae (\ra - \lae \Rb{t}) + \xb \Ibe (\rb - \lbe \Ra{t})\numberthis \label{eq:derivative}                                                                                    \\
                    & +\ra \Ia{t}(\ya + \yb - \xa) +\rb \Ib{t}(\ya + \yb - \xb)                                                                                                                       \\
                    & +\Ia{t}\left(\lae (\xa-\yb)\Rb{t}-\ya \frac{\ra \rb}{\lbe \Ra{t}}\right)                                                                                                        \\
                    & +\Ib{t}\left(\lbe (\xb-\ya)\Ra{t}-\yb \frac{\ra \rb}{\lae \Rb{t}}\right)                                                                                                        \\
    \end{align*}

    We now replace $\Ia{t}$ by $\numberOfVertices - \Ra{t}- \Rb{t}- \Ib{t}$ before taking the derivative by $\Ra{t}$ of \Cref{eq:derivative}. We get

    \begin{align*}
        \frac{d^2\Lyapunov_t}{dt \cdot d\Ra{t}}= & 0-\lbe \xb \Ibe                                                                                \\
                                     & -\ra (\ya+\yb-\xa)+0                                                                           \\
                                     & -\lae(\xa-\yb)\Rb{t}+\frac{\ya \ra \rb}{\lbe}\frac{\numberOfVertices- \Rb{t}-\Ib{t}}{\Ra{t}^2} \\
                                     & +\lbe (\xb-\ya)\Ib{t}+0.
    \end{align*}

    We now set this derivative to 0 and set $\Ra{t}=\Rae = \frac{\rb}{\lbe}$ and $\Rb{t}=\Rbe = \frac{\ra}{\lae}$. Further let $I^*=\numberOfVertices - \Rae - \Rbe$. Then $\Ibe=\frac{\ra}{\ra+\rb}I^*$. We get

    \begin{align*}
        0= & -\lbe \frac{\ra}{\ra+\rb}\xb I^* - \ra(\ya+\yb-\xa)-\ra(\xa-\yb)                                                                   \\
           & +\frac{\ya\ra \lbe}{\rb}(I^*-\Ib{t}+\frac{\rb}{\lbe}) + \lbe(\xb-\ya)\Ib{t}                                                        \\
        =  & (-\frac{\ra}{\rb}\ya+\xb-\ya)\lbe\Ib{t} +(-\frac{\ra}{\ra+\rb}\xb+\frac{\ra}{\rb}\ya)\lbe I^*\numberthis \label{eq:derivative_R1}.
    \end{align*}

    The right hand side of \Cref{eq:derivative_R1} is a linear equation of $\Ib{t}$. In order for this to be 0 for all $\Ib{t}$, both the coefficient of $\Ib{t}$ and the added constant have to be $0$. Both of those conditions are equivalent to $\xb = \frac{\ra +\rb}{\rb}\ya$ given that \lbe is positive. This concludes the proof.
\end{proof}

Doing the symmetric calculations for the derivative regarding $\Rb{t}$ also gives us the condition $\xa = \frac{\ra +\rb}{\ra}\yb$. Hence in order for the derivative of the Lyapunov function to be non-positive everywhere, the following condition needs to be met:
\begin{align}
    \xb = \frac{\ra +\rb}{\rb}\ya & \text{\quad and \quad}
    \xa = \frac{\ra +\rb}{\ra}\yb.\label{eq:condition_x}
\end{align}

Using these equations, we simplify $\frac{d \Lyapunov_t}{d t}$.

\begin{restatable}{lemma}{lastDegree}\label{lem:lastDegree}
    Let $G$ be a perfectly mixed graph with $\numberOfVertices \in \N$ vertices and density $\density\in (0,1]$ and let \contactProcess be a \IRIR process on $G$ with $\Rae + \Rbe <n$. Let \Cref{eq:condition_x} be true. Then
    \begin{align*}
        \frac{d \Lyapunov_t}{d t}= & -\lbe\frac{\ra}{\rb}\ya \da{t}\left(\Ia{t} \frac{\da{t}}{\Ra{t}}+(\da{t}+\db{t})\right)          \\
                      & -\lae\frac{\rb}{\ra}\yb \db{t}\left(\Ib{t} \frac{\db{t}}{\Rb{t}}+(\da{t}+\db{t})\right).\qedhere
    \end{align*}
\end{restatable}

\begin{proof}
    We start with \Cref{eq:derivative} and rewrite $\Ra{t}=\Rae+\da{t}$ and $\Rb{t}=\Rbe+\db{t}$ in order to calculate the derivative based on how far the state is away from the diagonal where it is zero. Recall that plugging $\Ra{t}=\Rae$ and $\Rb{t}=\Rbe$ into \Cref{eq:derivative} makes it zero for all \Ia{t} and \Ib{t}. Hence we can take \Cref{eq:derivative} and subtract from it \Cref{eq:derivative} with $\Ra{t}=\Rae$ and $\Rb{t}=\Rbe$. We get using $\frac{1}{x}-\frac{1}{x+c}=\frac{c}{x(c+x)}$

    \begin{align*}
        \frac{d \Lyapunov_t}{dt}= & -\lae \xa \Iae \db{t} -\lbe \xb \Ibe \da{t}                                                                                      \\
                    & +0+0                                                                                                                             \\
                    & +\lae (\xa-\yb)\Ia{t} \db{t} + \frac{\ra \ya \Ia{t} \da{t}}{\Ra{t}}                                                              \\
                    & +\lbe (\xb-\ya)\Ib{t} \da{t} + \frac{\rb \yb \Ib{t} \db{t}}{\Rb{t}}                                                              \\
        =           & \da{t}\left(-\lbe \xb \Ibe + \frac{\ra \ya \Ia{t}}{\Ra{t}}+\lbe (\xb-\ya)\Ib{t}\right)                                           \\
                    & + \db{t}\left(-\lae \xa \Iae + \frac{\rb \yb \Ib{t}}{\Rb{t}}+\lae (\xa-\yb)\Ia{t}\right)\numberthis \label{eq:derivative_delta}.
    \end{align*}

    For simplicity we now only rearrange the factor of $\da{t}$ as the other one is symmetric by swapping all 1's and 2's. We again use $I^*=\numberOfVertices - \Rae -\Rbe$ and replace $\Ib{t}= \numberOfVertices - \Ra{t} - \Rb{t}- \Ia{t}= I^* - \Ia{t} - (\da{t}+\db{t})$. We also use $\frac{1}{x+c}=\frac{1}{x}-\frac{c}{x(c+x)}$ again. Finally we use \Cref{eq:condition_x} to replace $\xb = \frac{\ra+\rb}{\rb}\ya$. We get

    \begin{align*}
        \da{t} & \left(-\lbe \xb \Ibe + \frac{\ra \ya \Ia{t}}{\Ra{t}}+\lbe (\xb-\ya)\Ib{t}\right)                                                                                        \\
        =      & \da{t}\left(-\lbe \frac{\ra}{\ra+\rb}\xb I^* + \frac{\ra \ya \Ia{t}}{\Rae+\da{t}}+\lbe (\xb-\ya)(I^* - \Ia{t} - (\da{t}+\db{t}))\right)                                 \\
        =      & \lbe \da{t}\left(- \frac{\ra}{\ra+\rb}\xb I^* + \frac{\ra \ya \Ia{t}}{\rb} - \frac{\ra \ya \Ia{t} \da{t}}{\rb \Ra{t}}+ (\xb-\ya)(I^* - \Ia{t} - (\da{t}+\db{t}))\right) \\
        =      & \lbe \da{t}\left(\Ia{t}(\frac{\ra}{\rb}\ya - \frac{\ra}{\rb}\ya \frac{\da{t}}{\Ra{t}}-\xb+\ya )-(\da{t}+\db{t})(\xb-\ya)+I^*(-\frac{\ra}{\ra+\rb}\xb+\xb-\ya)\right)    \\
        =      & \lbe \da{t}\left(\Ia{t}(- \frac{\ra}{\rb}\ya \frac{\da{t}}{\Ra{t}})-(\da{t}+\db{t})(\frac{\ra}{\rb}\ya)+I^*(0)\right)                                                   \\
        =      & -\lbe\frac{\ra}{\rb}\ya \da{t}\left(\Ia{t} \frac{\da{t}}{\Ra{t}}+(\da{t}+\db{t})\right)\numberthis\label{eq:derivative_delta_1}.
    \end{align*}

    Together with the symmetric calculation for the $\db{t}$ term that concludes the proof.\qedhere
\end{proof}

Note that multiplying all of \xa, \ya, \xb, and \yb by a positive number just multiplies the derivative of the Lyapunov function by that number as well. This does not change the fact whether the derivative is non-positive everywhere. Hence there will be infinitely many solutions for those parameters. Compressing those into one by normalizing in some way gets rid of one degree of freedom. Together with \Cref{eq:condition_x} that means that we have only one degree of freedom left.

Resolving the last degree of freedom by setting $\lbe\frac{\ra}{\rb}\ya= \lae\frac{\rb}{\ra}\yb$ makes all the terms in \Cref{lem:lastDegree} non-negative. We then normalize the parameters to make the terms easy to read and divide all of them by $n$ in order to make them constant for the $\lambda$'s and $\varrho$'s we usually consider. This gives us the following values for the parameters of $\Lyapunov_t$.
\begin{align*}
    \xa=\frac{\ra+\rb}{\rb\lae n}\text{;\quad} \ya=\frac{\rb}{\ra\lbe n} & \text{;\quad} \xb=\frac{\ra+\rb}{\ra\lbe n} \text{;\quad}
    \yb=\frac{\ra}{\rb\lae n}\numberthis\label{eq:parameters}.
\end{align*}

Setting the parameters in \Cref{lem:lastDegree} as in \Cref{eq:parameters} gives us.

\begin{restatable}{corollary}{derivativeF}\label{lem:derivativeF}
    Let $G$ be a perfectly mixed graph with $\numberOfVertices \in \N$ vertices and density $\density\in (0,1]$ and let \contactProcess be a \IRIR process on $G$ with $\Rae + \Rbe <n$. Let the parameters be chosen as in \Cref{eq:parameters}. Then
    \begin{align*}
        \frac{d \Lyapunov_t}{dt} & =-\frac{1}{n}\left(\frac{\Ia{t}}{\Ra{t}}\da{t}^2+\frac{\Ib{t}}{\Rb{t}}\db{t}^2+(\da{t}+\db{t})^2\right).\qedhere
    \end{align*}
\end{restatable}

\subsection{Shifting the $R$ values}\label{sec:shiftR}

We aim to use the derivative of the potential to apply drift arguments to get a super-polynomial survival time bound similar to \cite{FrGoKlKrPa2024}. To this end we want a high potential to correlate with states that are close to dying out. In the current potential $F$ that is not the case as the potential gets high when any of the four values $\Ia{t}, \Ib{t}, \Ra{t}, \Rb{t}$ becomes small. However small values of $\Ra{t}$ and $\Rb{t}$ do not mean that the infection is close to dying out. To change that, we adjust the potential in a similar way to \cite{FrGoKlKrPa2024} by shifting the values of $\Ra{t}$ and $\Rb{t}$ up by a constant. As $f(x',x)$ only becomes very big when $x$ goes to $0$, shifting both $x$ and $x'$ up by a constant keeps the desired properties while preventing it from becoming too big. We shift by $x'$ for convenience. Note that we calculated the coefficients before that made the drift non-positive everywhere by looking at the derivative around the equilibria of $\Ra{t}$ and $\Rb{t}$. Shifting them by their equilibrium value halves the derivative around the equilibrium, hence we have to multiply these terms by 2 in order to get similar results as before. We get the following adjusted potential function.

\begin{definition}\label{def:lyapunovPotential2}
    Let $G$ be a graph with $\numberOfVertices \in \N$ vertices and density $\density\in (0,1]$ and let \contactProcess be a \IRIR process on $G$ with $\Rae + \Rbe <n$.
    For all $t \in \R$, we define $\Lyapunovb_t$ as
    \begin{align*}
        \Lyapunovb_t & = \Lyapunovb(\Ia{t},\Ra{t},\Ib{t},\Rb{t})                                                                                                                   \\
                     & = \xa\lyapunovHelper[\Iae,\Ia{t}]+ 2\ya\lyapunovHelper[2\Rae,\Ra{t}+\Rae]+ \xb\lyapunovHelper[\Ibe,\Ib{t}]+ 2\yb\lyapunovHelper[2\Rbe,\Rb{t}+\Rbe].\qedhere
    \end{align*}
\end{definition}

We now calculate the derivative of $\Lyapunovb$ using $\da{t}$ and $\db{t}$.

\begin{restatable}{lemma}{derivativeK}\label{lem:derivativeK}
    Let $G$ be a perfectly mixed graph with $\numberOfVertices \in \N$ vertices and density $\density\in (0,1]$ and let \contactProcess be a \IRIR process on $G$ with $\Rae + \Rbe <n$. Let the parameters be chosen as in \Cref{eq:parameters}. Then
    \begin{align*}
        \frac{d \Lyapunovb_t}{dt} & =-\frac{1}{n}\left(\frac{\da{t}^2}{\Ra{t}+\rb/\lbe}\left(\Ia{t}+\frac{\rb}{\ra}\Ib{t}\right)+\frac{\db{t}^2}{\Rb{t}+\ra/\lae}\left(\Ib{t}+\frac{\ra}{\rb}\Ia{t}\right)+(\da{t}+\db{t})^2\right).\qedhere
    \end{align*}
\end{restatable}

\begin{proof}

    By definition of $\Lyapunovb$, we get

    \begin{align*}
        \frac{d \Lyapunovb_t}{d t}= & \xa(1-\frac{\Iae}{\Ia{t}})(\lae \Ia{t} \Rb{t}- \ra \Ia{t})                                                                                                       \\
                       & +2\ya(1-\frac{2\Rae}{\Ra{t}+\Rae})(\ra \Ia{t}-\lbe \Ib{t}\Ra{t})                                                                                                 \\
                       & +\xb(1-\frac{\Ibe}{\Ib{t}})(\lbe \Ib{t} \Ra{t}- \rb \Ib{t})                                                                                                      \\
                       & +2\yb(1-\frac{2\Rbe}{\Rb{t}+\Rbe})(\rb \Ib{t}-\lae \Ia{t}\Rb{t})                                                                                                 \\
        =              & \xa(\Ia{t}-\Iae)(\lae \db{t})                                                                                                                                    \\
                       & +2\ya(\frac{\da{t}}{\Ra{t}+\rb/\lbe})(\ra \Ia{t}-\lbe \Ib{t}\Ra{t})                                                                                              \\
                       & +\xb(\Ib{t}-\Ibe)(\lbe \da{t})                                                                                                                                   \\
                       & +2\yb(\frac{\db{t}}{\Rb{t}+\ra/\lae})(\rb \Ib{t}-\lae \Ia{t}\Rb{t})                                                                                              \\
        =              & \xa \lae \db{t} \Ia{t} - \xa \lae \Iae \db{t}                                                                                                                    \\
                       & +2\ya \ra \frac{\da{t}\Ia{t}}{\Ra{t}+\rb/\lbe} - 2\ya \lbe \frac{\da{t} \Ib{t} \Ra{t}}{\Ra{t}+\rb/\lbe}                                                          \\
                       & +\xb \lbe \da{t} \Ib{t} - \xb \lbe \Ibe \da{t}                                                                                                                   \\
                       & +2\yb \rb \frac{\db{t}\Ib{t}}{\Rb{t}+\ra/\lae} - 2\yb \lae \frac{\db{t} \Ia{t} \Rb{t}}{\Rb{t}+\ra/\lae}                                                          \\
        =              & \da{t}\left(\xb \lbe(\Ib{t}-\Ibe)+2\ya\ra\frac{\Ia{t}}{\Ra{t}+\rb/\lbe}-2\ya\lbe\frac{\Ib{t}\Ra{t}}{\Ra{t}+\rb/\lbe}\right)                                      \\
                       & +\db{t}\left(\xa \lae(\Ia{t}-\Iae)+2\yb\rb\frac{\Ib{t}}{\Rb{t}+\ra/\lae}-2\yb\lae\frac{\Ia{t}\Rb{t}}{\Rb{t}+\ra/\lae}\right).\numberthis\label{eq:Lyapunovb_pre}
    \end{align*}

    We simplify the factor of $\da{t}$ as the other one follows by symmetry. To this end we plug in the values from \Cref{eq:parameters} and use $I^*=\numberOfVertices - \Rae -\Rbe$ and replace $\Ib{t}= \numberOfVertices - \Ra{t} - \Rb{t}- \Ia{t}= I^* - \Ia{t} - (\da{t}+\db{t})$. Note that $\Ibe = \frac{\ra}{\ra+\rb}I^*$. We get

    \begin{align*}
         & \da{t}\left(\xb \lbe(\Ib{t}-\Ibe)+2\ya\ra\frac{\Ia{t}}{\Ra{t}+\rb/\lbe}-2\ya\lbe\frac{\Ib{t}\Ra{t}}{\Ra{t}+\rb/\lbe}\right)                                                              \\
         & =\frac{\da{t}}{n}\left(\frac{\ra+\rb}{\ra}(\Ib{t}-\Ibe)+2\frac{\rb}{\lbe}\frac{\Ia{t}}{\Ra{t}+\rb/\lbe}-2\frac{\rb}{\ra}\frac{\Ib{t}\Ra{t}}{\Ra{t}+\rb/\lbe}\right)                      \\
         & =\frac{\da{t}}{n}\left(\frac{\ra+\rb}{\ra}(\Ib{t}-\Ibe)+2\frac{\rb}{\lbe}\frac{\Ia{t}}{\Ra{t}+\rb/\lbe}-\frac{\rb}{\ra}\Ib{t}-\frac{\rb}{\ra}\frac{\Ib{t}\da{t}}{\Ra{t}+\rb/\lbe}\right) \\
         & =\frac{\da{t}}{n}\left(\Ib{t}-I^*+2\frac{\rb}{\lbe}\frac{\Ia{t}}{\Ra{t}+\rb/\lbe}-\frac{\rb}{\ra}\frac{\Ib{t}\da{t}}{\Ra{t}+\rb/\lbe}\right)                                             \\
         & =\frac{\da{t}}{n}\left(I^*-\Ia{t}-(\da{t}+\db{t})-I^*+2\frac{\rb}{\lbe}\frac{\Ia{t}}{\Ra{t}+\rb/\lbe}-\frac{\rb}{\ra}\frac{\Ib{t}\da{t}}{\Ra{t}+\rb/\lbe}\right)                         \\
         & =\frac{\da{t}}{n}\left(-(\da{t}+\db{t})-\frac{\Ia{t}\da{t}}{\Ra{t}+\rb/\lbe}-\frac{\rb}{\ra}\frac{\Ib{t}\da{t}}{\Ra{t}+\rb/\lbe}\right)                                                  \\
    \end{align*}

    Plugging that into \Cref{eq:Lyapunovb_pre} gives us

    \begin{align*}
        \frac{d \Lyapunovb_t}{dt} & =-\frac{1}{n}\left(\frac{\da{t}^2}{\Ra{t}+\rb/\lbe}\left(\Ia{t}+\frac{\rb}{\ra}\Ib{t}\right)+\frac{\db{t}^2}{\Rb{t}+\ra/\lae}\left(\Ib{t}+\frac{\ra}{\rb}\Ia{t}\right)+(\da{t}+\db{t})^2\right).\qedhere
    \end{align*}
\end{proof}

Note that this derivative is still non-negative for all states with a positive number of vertices in each state. It is also still 0 whenever both $\Ra{t}$ and $\Rb{t}$ are at their equilibrium value.

\subsection{Changing the $R$ target}\label{sec:changeR}

We aim to show long survival times using a negative drift theorem. To this end we need a region in the potential space in which the process has a constant negative drift. Currently the derivative is 0 whenever both $\Ra{t}$ and $\Rb{t}$ are at their equilibrium value (which we call the diagonal). This makes it impossible to find such a region. The idea of how to solve that issue is to change the potential slightly such that the drift becomes negative at the diagonal and stays negative everywhere else. At the diagonal, the derivative of $\Ia{t}$ and $\Ib{t}$ by $t$ is 0. Hence we have to adjust the terms for $\Ra{t}$ and $\Rb{t}$. At the diagonal, when $\Ia{t}$ is above its equilibrium and $\Ib{t}$ below, $\Ra{t}$ tends to increase and $\Rb{t}$ tends to decrease. To use that fact, we adjust the target for $\Rb{t}$ slightly such that is a bit bigger if $\Ia{t}$ is below its equilibrium and higher if $\Ia{t}$ is above its equilibrium. We do the symmetric adjustment for $\Ra{t}$. We get the following potential function.

\begin{definition}\label{def:lyapunovPotential3}
    Let $G$ be a graph with $\numberOfVertices \in \N$ vertices and density $\density\in (0,1]$ and let \contactProcess be a \IRIR process on $G$ with $\Rae + \Rbe <n$.
    For all $t \in \R$, we define $\Lyapunovc_t$ as
    \begin{align*}
        \Lyapunovc_t = & \Lyapunovc(\Ia{t},\Ra{t},\Ib{t},\Rb{t})                                                                           \\
        =              & \xa\lyapunovHelper[\Iae,\Ia{t}]+ 2\ya\lyapunovHelper[(1+\frac{\ln(\Ibe)}{\ln(\Ib{t})})\Rae,\Ra{t}+\Rae]           \\
                       & +\xb\lyapunovHelper[\Ibe,\Ib{t}]+ 2\yb\lyapunovHelper[(1+\frac{\ln(\Iae)}{\ln(\Ia{t})})\Rbe,\Rb{t}+\Rbe].\qedhere
    \end{align*}
\end{definition}

The problem with this function is that some terms now depend on two variables. Hence, taking the derivative by a variable is not only the derivative of its one term anymore. To calculate the drift, we first calculate the derivative of the $R$ terms by the $I$ term that appears in them. We do this for $\Ra{t}$ and take the derivative by $\Ib{t}$, the other one follows by symmetry. We get

\begin{restatable}{lemma}{derivativeRI}\label{lem:derivativeRI}
    Let $G$ be a graph with $\numberOfVertices \in \N$ vertices and density $\density\in (0,1]$ and let \contactProcess be a \IRIR process on $G$ with $\Rae + \Rbe <n$. Let the parameters be chosen as in \Cref{eq:parameters}. Then
    \begin{align*}
        \frac{d\lyapunovHelper[(1+\frac{\ln(\Ibe)}{\ln(\Ib{t})})\Rae,\Ra{t}+\Rae]}{d \Ib{t}} & =\frac{\Rae \ln(\Ibe)}{\Ib{t} \ln^2(\Ib{t})}\ln(\frac{\Ra{t}+\Rae}{\Rae(1+\frac{\ln(\Ibe)}{\ln(\Ib{t})})}).\qedhere
    \end{align*}
\end{restatable}

\begin{proof}
    The lemma follows using common derivative laws.

    \begin{align*}
          & \frac{d\lyapunovHelper[(1+\frac{\ln(\Ibe)}{\ln(\Ib{t})})\Rae,\Ra{t}+\Rae]}{d \Ib{t}}                                              \\
        = & \left(\Ra{t}+\Rae-\Rae(1+\frac{\ln(\Ibe)}{\ln(\Ib{t})})(\ln(\frac{\Ra{t}+\Rae}{\Rae(1+\frac{\ln(\Ibe)}{\ln(\Ib{t})})})+1)\right)' \\
        = & -\left(\Rae(1+\frac{\ln(\Ibe)}{\ln(\Ib{t})})\right)'(\ln(\frac{\Ra{t}+\Rae}{\Rae(1+\frac{\ln(\Ibe)}{\ln(\Ib{t})})})+1)            \\
          & -\Rae(1+\frac{\ln(\Ibe)}{\Ib{t}})\left(\ln(\frac{\Ra{t}+\Rae}{\Rae(1+\frac{\ln(\Ibe)}{\ln(\Ib{t})})})+1\right)'                   \\
        = & \frac{\Rae \ln(\Ibe)}{\Ib{t} \ln^2(\Ib{t})}(\ln(\frac{\Ra{t}+\Rae}{\Rae(1+\frac{\ln(\Ibe)}{\ln(\Ib{t})})})+1)                     \\
          & -\Rae(1+\frac{\ln(\Ibe)}{\ln(\Ib{t})})\frac{\ln(\Ibe)}{\Ib{t}\ln^2(\Ib{t})(1+\ln(\Ibe)/\ln(\Ib{t}))}                              \\
        = & \frac{\Rae \ln(\Ibe)}{\Ib{t} \ln^2(\Ib{t})}\ln(\frac{\Ra{t}+\Rae}{\Rae(1+\frac{\ln(\Ibe)}{\ln(\Ib{t})})}).\qedhere
    \end{align*}
\end{proof}

Using that, we calculate the derivative of $L$ in relation to the derivative of $K$ in \Cref{lem:derivativeL}.

\begin{lemma}\label{lem:derivativeL}
    Let $G$ be a perfectly mixed graph with $\numberOfVertices \in \N$ vertices and density $\density\in (0,1]$ and let \contactProcess be a \IRIR process on $G$ with $\Rae + \Rbe <n$. Let the parameters be chosen as in \Cref{eq:parameters}. Then

    \begin{align*}
        \frac{d \Lyapunovc_t}{dt}= & \frac{d \Lyapunovb_t}{dt}                                                                                                                       \\
                       & +2\yb\frac{\Rbe \ln(\Iae)}{ \ln^2(\Ia{t})}\ln(\frac{\Rb{t}+\Rbe}{\Rbe(1+\frac{\ln(\Iae)}{\ln(\Ia{t})})})(\lae\Rb{t}-\ra) \\
                       & +2\ya\frac{\Rae(1-\frac{\ln(\Ibe)}{\ln(\Ib{t})})}{\Ra{t}+\Rae}(\ra\Ia{t}-\lbe\Ib{t}\Ra{t})                               \\
                       & +2\ya\frac{\Rae \ln(\Ibe)}{ \ln^2(\Ib{t})}\ln(\frac{\Ra{t}+\Rae}{\Rae(1+\frac{\ln(\Ibe)}{\ln(\Ib{t})})})(\lbe\Ra{t}-\rb) \\
                       & +2\yb\frac{\Rbe(1-\frac{\ln(\Iae)}{\ln(\Ia{t})})}{\Rb{t}+\Rbe}(\rb\Ib{t}-\lae\Ia{t}\Rb{t}).\qedhere
    \end{align*}
\end{lemma}

\begin{proof}
    The proof follows from calculating the derivative of $L$ as the sum of the derivatives by the $I$'s and $R$'s times their derivatives. For the derivative by the $I$'s also refer to \Cref{lem:derivativeRI}. We then use \Cref{lem:derivativeK} to isolate $\frac{d \Lyapunovb_t}{dt}$.

    \begin{align*}
        \frac{d \Lyapunovc_t}{dt}= & \left(\xa(1-\frac{\Iae}{\Ia{t}})+2\yb\frac{\Rbe \ln(\Iae)}{\Ia{t} \ln^2(\Ia{t})}\ln(\frac{\Rb{t}+\Rbe}{\Rbe(1+\frac{\ln(\Iae)}{\ln(\Ia{t})})})\right)(\lae\Ia{t}\Rb{t}-\ra\Ia{t})  \\
            & +2\ya(1-\frac{\Rae(1+\frac{\ln(\Ibe)}{\ln(\Ib{t})})}{\Ra{t}+\Rae})(\ra\Ia{t}-\lbe\Ib{t}\Ra{t})                                                                                     \\
            & +\left(\xb(1-\frac{\Ibe}{\Ib{t}})+2\ya\frac{\Rae \ln(\Ibe)}{\Ib{t} \ln^2(\Ib{t})}\ln(\frac{\Ra{t}+\Rae}{\Rae(1+\frac{\ln(\Ibe)}{\ln(\Ib{t})})})\right)(\lbe\Ib{t}\Ra{t}-\rb\Ib{t}) \\
            & +2\yb(1-\frac{\Rbe(1+\frac{\ln(\Iae)}{\ln(\Ia{t})})}{\Rb{t}+\Rbe})(\rb\Ib{t}-\lae\Ia{t}\Rb{t})                                                                                     \\
        =   & \frac{d \Lyapunovb_t}{dt}                                                                                                                                                                                 \\
            & +2\yb\frac{\Rbe \ln(\Iae)}{ \ln^2(\Ia{t})}\ln(\frac{\Rb{t}+\Rbe}{\Rbe(1+\frac{\ln(\Iae)}{\ln(\Ia{t})})})(\lae\Rb{t}-\ra)                                                           \\
            & +2\ya\frac{\Rae(1-\frac{\ln(\Ibe)}{\ln(\Ib{t})})}{\Ra{t}+\Rae}(\ra\Ia{t}-\lbe\Ib{t}\Ra{t})                                                                                         \\
            & +2\ya\frac{\Rae \ln(\Ibe)}{ \ln^2(\Ib{t})}\ln(\frac{\Ra{t}+\Rae}{\Rae(1+\frac{\ln(\Ibe)}{\ln(\Ib{t})})})(\lbe\Ra{t}-\rb)                                                           \\
            & +2\yb\frac{\Rbe(1-\frac{\ln(\Iae)}{\ln(\Ia{t})})}{\Rb{t}+\Rbe}(\rb\Ib{t}-\lae\Ia{t}\Rb{t}).\qedhere
    \end{align*}
\end{proof}

\section{Bounding the Survival Times}

In this section we derive survival time bounds for the IRIR process. To this end we aim to apply a drift theorem on the potential $\Lyapunovc$. For a detailed overview on which properties are needed for that, refer to \Cref{sec:drift}. We start by showing that we have a band in the potential space in which we get some bounds on the $I$ values. The following lemma shows that a small potential implies that both of the $I$ values are relatively big.

\begin{lemma}\label{lem:potentialBothBig}
    Let $C$ be a \IRIR process on a graph $G$ with $n$ vertices and density $\density \in (0,1]$. Let $C$ have constant recovery rates $\ra,\rb \in \R_{>0}$ and effective infection rates $\lae=\ca/n$ and $\lbe=\cb/n$ for constants $\ca,\cb \in \R_{>0}$. Let there be a constant $c\in \R_{>0}$ such that $\Rae + \Rbe < (1-c)n$.

    Let $\varepsilon_l \in \R_{>0}$ be a constant and let

    $$c_h \coloneqq \min\left(\frac{c}{\ca}\left(\ln(\frac{1}{\varepsilon_l})-\ln(\frac{\ra+\rb}{\rb c})-1\right), \frac{c}{\cb}\left(\ln(\frac{1}{\varepsilon_l})-\ln(\frac{\ra+\rb}{\ra c})-1\right)\right).$$

    If $\Lyapunovc_t<c_h n$ then $\Ia{t} > \varepsilon_l n$ and $\Ib{t} > \varepsilon_l n$.
\end{lemma}

\begin{proof}
    We show the contraposition of the statement. So we show that if $\Ia{t} \leq \varepsilon_l n$ or $\Ib{t} \leq \varepsilon_l n$ then $\Lyapunovc_t\geq c_h n$. Assume that $\Ia{t} \leq \varepsilon_l n$, the other case follows using symmetric argumentation.

    First note that $\lyapunovHelper$ is non-negative. Hence
    \begin{align*}
        \Lyapunovc & \geq \xa \lyapunovHelper(\Iae,\Ia{t})                                                      \\
                   & =\xa\Iae\left(\frac{\Ia{t}}{\Iae}-\ln(\frac{\Ia{t}}{\Iae})-1\right)                        \\
                   & =\frac{\ra+\rb}{\rb \lae n}\Iae\left(\frac{\Ia{t}}{\Iae}-\ln(\frac{\Ia{t}}{\Iae})-1\right) \\
                   & =\frac{I^*}{\ca}\left(\frac{\Ia{t}}{\Iae}-\ln(\frac{\Ia{t}}{\Iae})-1\right)                \\
                   & \geq\frac{c}{\ca}n\left(\frac{\Ia{t}}{\Iae}-\ln(\frac{\Ia{t}}{\Iae})-1\right)              \\
                   & \geq \frac{c}{\ca}n\left(0-\ln(\frac{\varepsilon_l n}{\Iae})-1\right)                      \\
                   & =\frac{c}{\ca}n\left(\ln(\frac{1}{\varepsilon_l})-\ln(\frac{n}{\Iae})-1\right)             \\
                   & \geq\frac{c}{\ca}n\left(\ln(\frac{1}{\varepsilon_l})-\ln(\frac{\ra+\rb}{\rb c})-1\right)   \\
                   & \geq c_h n.\qedhere
    \end{align*}
\end{proof}

The next lemma shows the opposite direction to the previous one. To be precise, it shows that a big potential implies that one of the $I$ values is very small.

\begin{lemma}\label{lem:potentialOneSmall}
    Let $C$ be a \IRIR process on a graph $G$ with $n$ vertices and density $\density \in (0,1]$. Let $C$ have constant recovery rates $\ra,\rb \in \R_{>0}$ and effective infection rates $\lae=\ca/n$ and $\lbe=\cb/n$ for constants $\ca,\cb \in \R_{>0}$. Let there be a constant $c\in \R_{>0}$ such that $\Rae + \Rbe < (1-c)n$.

    Let $\varepsilon_h \in \R_{>0}$ be a constant. Then there exists a constant $c_l \in \R_{>0}$ such that if $\Lyapunovc_t>c_l n$ then $\Ia{t} < \varepsilon_h n$ or $\Ib{t} < \varepsilon_h n$.
\end{lemma}

\begin{proof}
    We show the contraposition of the statement. So we show that if $\Ia{t} \geq \varepsilon_h n$ and $\Ib{t} \geq \varepsilon_h n$ then $\Lyapunovc_t\leq c_l n$. There are four terms in $\Lyapunovc_t$. We show for each of them that they are smaller than $\frac{c_l}{4} n$ for $c_l$ chosen high enough.

    We first upper bound the $\Ia{t}$ term and note that the $\Ib{t}$ term can be bound similarly. We know

    \begin{align*}
             & \xa \lyapunovHelper(\Iae,\Ia{t})                                                          \\
        =    & \xa\Iae\left(\frac{\Ia{t}}{\Iae}-\ln(\frac{\Ia{t}}{\Iae})-1\right)                        \\
        =    & \frac{\ra+\rb}{\rb \lae n}\Iae\left(\frac{\Ia{t}}{\Iae}-\ln(\frac{\Ia{t}}{\Iae})-1\right) \\
        =    & \frac{I^*}{\ca}\left(\frac{\Ia{t}}{\Iae}-\ln(\frac{\Ia{t}}{\Iae})-1\right)                \\
        \leq & \frac{n}{\ca}\left(\frac{n}{c \rb n/(\ra+\rb)}-\ln(\frac{\varepsilon_h n}{n})-1\right)    \\
        \leq & \frac{n}{\ca}\left(\frac{\ra+\rb}{c \rb}+\ln(\frac{1}{\varepsilon_h})-1\right)            \\
        \leq & \frac{c_l}{4}n.
    \end{align*}

    In order to bound the other terms, we first bound $l_2 = 1 + \frac{\ln(\Ibe)}{\ln(\Ib{t})}$. As both $\Ibe$ and $\Ib{t}$ are bigger than 1, the logarithms are positive and $l_2>1$. For sufficiently large $n$ we have $\Ibe \leq n < \varepsilon_h^2 n^2 \leq \Ib{t}^2$. Hence $l_2<3$. With this we bound the $\Ra{t}$ term of $\Lyapunovc_t$, noting that the $\Rb{t}$ term can be bound in the same way.

    \begin{align*}
             & 2\ya \lyapunovHelper(l_2 \Rae,\Ra{t}+\Rae)                                                  \\
        =    & 2 \ya l_2 \Rae\left(\frac{\Ra{t}+\Rae}{l_2 \Rae}-\ln(\frac{\Ra{t}+\Rae}{l_2 \Rae})-1\right) \\
        \leq & 2\frac{\rb}{\ra \cb}3n\left(\frac{2n}{\frac{\rb}{\cb}n}-\ln(\frac{\Rae}{3\Rae})-1\right)    \\
        =    & 6\frac{\rb}{\ra \cb}n\left(2\frac{\cb}{\rb}+\ln(3)-1\right)                                 \\
        \leq & \frac{c_l}{4}n.\qedhere
    \end{align*}
\end{proof}

In order to get from the mean-field analysis to an analysis of the original stochastic process, we need to convert the derivatives to steps. The following lemma shows that steps that increase $\Ib{t}$ change the potential term of $\Ra{t}$ by as much as the derivative would suggest except for lower order terms.

\begin{restatable}{lemma}{derivativeDiscrete}\label{lem:derivativeDiscrete}
    Let $G$ be a graph with $\numberOfVertices$ vertices and density $\density \in (0,1]$ and let \contactProcess be a \IRIR process on $G$ with $\Rae + \Rbe <(1-c)n$ for a constant $c\in \R_{>0}$ and with constant recovery rates and effective infection rates in $\bigTheta{n^{-1}}$. Let the parameters be chosen as in \Cref{eq:parameters} and let $\Ib{t} \in \bigTheta{n}$. Then

    \begin{align*}
          & \lyapunovHelper[(1+\frac{\ln(\Ibe)}{\ln(\Ib{t}+1)})\Rae,\Ra{t}+\Rae]-\lyapunovHelper[(1+\frac{\ln(\Ibe)}{\ln(\Ib{t})})\Rae,\Ra{t}+\Rae] \\
        = & \frac{d\lyapunovHelper[(1+\frac{\ln(\Ibe)}{\ln(\Ib{t})})\Rae,\Ra{t}+\Rae]}{d \Ib{t}}\pm \bigO{n^{-1}}.\qedhere
    \end{align*}
\end{restatable}

\begin{proof}
    First note that with the given parameter restrictions, all of $\Ib{t},\Rae,\Ibe$ are in $\bigTheta{n}$. Using known bounds for the natural logarithm, we get that for any $x\in \R_{>-0.68}$ it holds $x-x^2 \leq \ln(1+x) \leq x$. In particular if $x \in \bigO{n^{-1}}$, we get that $\ln(1+x)= x - \bigO{n^{-2}}$. In particular

    \begin{align}
        \ln(\Ib{t}+1)-\ln(\Ib{t}) & = \ln(\frac{\Ib{t}+1}{\Ib{t}})=\ln(1+\frac{1}{\Ib{t}})= \frac{1}{\Ib{t}} - \bigO{n^{-2}}.\label{eq:lnDiff}
    \end{align}

    This also implies that

    \begin{align*}
        \frac{1}{\ln(\Ib{t}+1)\ln(\Ib{t})} & = \frac{1}{\ln^2(\Ib{t})}\cdot \frac{\ln(\Ib{t})}{\ln(\Ib{t}+1)}                  \\
                                           & =\frac{1}{\ln^2(\Ib{t})}\cdot (1-\frac{\ln(\Ib{t}+1)-\ln(\Ib{t})}{\ln(\Ib{t}+1)}) \\
                                           & =\frac{1}{\ln^2(\Ib{t})}\cdot(1-\bigO{n^{-1}}).\numberthis\label{eq:lnProd}
    \end{align*}

    Now consider the equation we aim to show. We calculated the derivative in \Cref{lem:derivativeRI}. We start rearranging the other term by using the definition of $\lyapunovHelper$ and canceling out terms that are the same on both sides of the subtraction.

    \begin{align*}
          & \lyapunovHelper[(1+\frac{\ln(\Ibe)}{\ln(\Ib{t}+1)})\Rae,\Ra{t}+\Rae]-\lyapunovHelper[(1+\frac{\ln(\Ibe)}{\ln(\Ib{t})})\Rae,\Ra{t}+\Rae]                  \\
        = & \Rae(1+\frac{\ln(\Ibe)}{\ln(\Ib{t})})\left(\ln(\frac{\Ra{t}+\Rae}{\Rae(1+\frac{\ln(\Ibe)}{\ln(\Ib{t})})})+1\right)                                       \\
          & -\Rae(1+\frac{\ln(\Ibe)}{\ln(\Ib{t}+1)})\left(\ln(\frac{\Ra{t}+\Rae}{\Rae(1+\frac{\ln(\Ibe)}{\ln(\Ib{t}+1)})})+1\right).\numberthis\label{eq:bigFormula}
    \end{align*}

    The strategy is now to write the factors in the second term as the factors of the first term plus the appropriate difference. Multiplying things out than cancels the entire first term and just leaves us with the difference. We first calculate the difference of the second terms using \Cref{eq:lnDiff} and \Cref{eq:lnProd}.

    \begin{align*}
          & (1+\frac{\ln(\Ibe)}{\ln(\Ib{t}+1)})-(1+\frac{\ln(\Ibe)}{\ln(\Ib{t})})                   \\
        = & \frac{\ln(\Ibe)}{\ln(\Ib{t}+1)}-\frac{\ln(\Ibe)}{\ln(\Ib{t})}                           \\
        = & \frac{\ln(\Ibe)(\ln(\Ib{t})-\ln(\Ib{t}+1))}{\ln(\Ib{t})\ln(\Ib{t}+1)}                   \\
        = & \frac{\ln(\Ibe)}{\ln^2(\Ib{t})}(1\pm \bigO{n^{-1}})(-\frac{1}{\Ib{t}}\pm \bigO{n^{-2}}) \\
        = & -\frac{\ln(\Ibe)}{\Ib{t}\ln^2(\Ib{t})}\pm \bigO{n^{-2}}.\numberthis\label{eq:termTwo}
    \end{align*}

    Doing similar calculation for the third term, we get

    \begin{align*}
          & \left(\ln(\frac{\Ra{t}+\Rae}{\Rae(1+\frac{\ln(\Ibe)}{\ln(\Ib{t}+1)})})+1\right)-\left(\ln(\frac{\Ra{t}+\Rae}{\Rae(1+\frac{\ln(\Ibe)}{\ln(\Ib{t})})})+1\right) \\
        = & -\ln\left(\frac{1+\frac{\ln(\Ibe)}{\ln(\Ib{t}+1)}}{1+\frac{\ln(\Ibe)}{\ln(\Ib{t})}}\right)                                                                    \\
        = & -\ln\left(1+\frac{\frac{\ln(\Ibe)}{\ln(\Ib{t}+1)}-\frac{\ln(\Ibe)}{\ln(\Ib{t})}}{1+\frac{\ln(\Ibe)}{\ln(\Ib{t})}}\right)                                      \\
        = & -\ln\left(1+\frac{\frac{\ln(\Ibe)(\ln(\Ib{t})-\ln(\Ib{t}+1))}{\ln(\Ib{t}+1)\ln(\Ib{t})}}{1+\frac{\ln(\Ibe)}{\ln(\Ib{t})}}\right)                              \\
        = & -\ln\left(1+\frac{\ln(\Ibe)(-1/\Ib{t} \pm \bigO{n^{-2}})}{\ln^2(\Ib{t})(1+\frac{\ln(\Ibe)}{\ln(\Ib{t})})}(1-\bigO{n^{-1}})\right)                             \\
        = & -\ln\left(1-\frac{\ln(\Ibe)}{\Ib{t}\ln^2(\Ib{t})(1+\frac{\ln(\Ibe)}{\ln(\Ib{t})})}\pm \bigO{n^{-2}}\right)                                                    \\
        = & \frac{\ln(\Ibe)}{\Ib{t}\ln^2(\Ib{t})(1+\frac{\ln(\Ibe)}{\ln(\Ib{t})})}\pm \bigO{n^{-2}}.\numberthis\label{eq:termThree}
    \end{align*}

    We aim to plug \Cref{eq:termTwo} and \Cref{eq:termThree} into \Cref{eq:bigFormula} to replace two of the factors on the right. That makes it into a formula of the form

    \begin{align*}
        xyz & - x(y+y'\pm \bigO{n^{-2}})(z+z'\pm \bigO{n^{-2}}).
    \end{align*}

    Note that in this case $x\in \bigTheta{n}$, $y,z \in \bigO{1}$ and $y',z' \in \bigO{n^{-1}}$. Hence multiplying the terms out and grouping all resulting terms in $\bigO{n^{-1}}$ together gives

    \begin{align*}
        -xyz' -xy'z \pm \bigO{n^{-1}}.
    \end{align*}

    Actually applying that to \Cref{eq:bigFormula} gives

    \begin{align*}
          & \Rae(1+\frac{\ln(\Ibe)}{\ln(\Ib{t})})\left(\ln(\frac{\Ra{t}+\Rae}{\Rae(1+\frac{\ln(\Ibe)}{\ln(\Ib{t})})})+1\right)          \\
          & -\Rae(1+\frac{\ln(\Ibe)}{\ln(\Ib{t}+1)})\left(\ln(\frac{\Ra{t}+\Rae}{\Rae(1+\frac{\ln(\Ibe)}{\ln(\Ib{t}+1)})})+1\right)     \\
        = & -\Rae(1+\frac{\ln(\Ibe)}{\ln(\Ib{t})})\frac{\ln(\Ibe)}{\Ib{t}\ln^2(\Ib{t})(1+\frac{\ln(\Ibe)}{\ln(\Ib{t})})}                \\
          & +\Rae \frac{\ln(\Ibe)}{\Ib{t}\ln^2(\Ib{t})}\left(\ln(\frac{\Ra{t}+\Rae}{\Rae(1+\frac{\ln(\Ibe)}{\ln(\Ib{t})})})+1\right)    \\
          & \pm \bigO{n^{-1}}                                                                                                           \\
        = & \Rae \frac{\ln(\Ibe)}{\Ib{t}\ln^2(\Ib{t})}\ln(\frac{\Ra{t}+\Rae}{\Rae(1+\frac{\ln(\Ibe)}{\ln(\Ib{t})})}) \pm \bigO{n^{-1}}.
    \end{align*}

    Noting that the last term is exactly the derivative from \Cref{lem:derivativeRI} except for the $\pm \bigO{n^{-1}}$ concludes the proof.
\end{proof}

As a similar statement to \Cref{lem:derivativeDiscrete} also holds for the change in the $\Ib{t}$ term of the potential, we get that the potential change of a step of size 1 equals the derivative by the changed value plus lower order terms.

\begin{lemma}\label{lem:differenceDerivative}
    Let $C$ be a \IRIR process on a graph $G$ with $n$ vertices and density $\density \in (0,1]$. Let $C$ have constant recovery rates $\ra,\rb \in \R_{>0}$ and effective infection rates $\lae=\ca/n$ and $\lbe=\cb/n$ for constants $\ca,\cb \in \R_{>0}$. Let $\varepsilon_l$ be a constant and let $i \in \N$ with $t\coloneqq \timeContinuous{i}$ such that $\Ia{t}\geq \varepsilon_l n$ and $\Ib{t}\geq \varepsilon_l n$. Then

    \begin{align*}
          & \Lyapunovc(\Ia{t},\Ra{t},\Ib{t}+1,\Rb{t})-\Lyapunovc(\Ia{t},\Ra{t},\Ib{t},\Rb{t}) \\
        = & \frac{d \Lyapunovc(\Ia{t},\Ra{t},\Ib{t},\Rb{t})}{d \Ib{t}} \pm \bigO{n^{-1}}.
    \end{align*}

    Equivalent statements also hold for variable changes and derivatives by $\Ia{t}$, $\Ra{t}$ and $\Rb{t}$ respectively.
\end{lemma}

\begin{proof}
    The lemma statement follows directly from \Cref{lem:lyapunovHelper} and \Cref{lem:derivativeDiscrete}. The difference of the $L$ values can be written as a sum of the difference of the corresponding $f$ values. The $\Ia{t}$ and $\Rb{t}$ terms are unaffected by changing $\Ib{t}$ and hence the difference is 0. The $\Ib{t}$ term difference is by \Cref{lem:lyapunovHelper} only off from the derivative of that term by $\bigO{n^{-1}}$ as $\Ib{t} \in \bigTheta{n}$. The same is true for the $\Ra{t}$ term by \Cref{lem:derivativeDiscrete}. As the derivative of $L$ is the sum of derivatives of the $f$ terms, the lemma statement follows.

    The argumentation for $\Ia{t}$ works the exact same way. For $\Ra{t}$ and $\Rb{t}$ only one of the $f$ derivatives is non-zero, hence by \Cref{lem:lyapunovHelper} alone the statement follows.
\end{proof}

We now show that the absolute value of the derivative of $\Lyapunovc$ by any of the $R's$ and $I's$ is upper bounded by a constant in the region we care for.

\begin{lemma}\label{lem:constantStep}
    Let $C$ be a \IRIR process on a graph $G$ with $n$ vertices and density $\density \in (0,1]$. Let $C$ have constant recovery rates $\ra,\rb \in \R_{>0}$ and effective infection rates $\lae=\ca/n$ and $\lbe=\cb/n$ for constants $\ca,\cb \in \R_{>0}$. Let there be a constant $c\in \R_{>0}$ such that $\Rae + \Rbe < (1-c)n$. Let $\varepsilon_l$ be a constant and let $i \in \N$ with $t\coloneqq \timeContinuous{i}$ such that $\Ia{t}\geq \varepsilon_l n$ and $\Ib{t}\geq \varepsilon_l n$. Then there exists a constant $c_s\in \R_{>0}$ such that

    $$\max\left(\left\lvert\frac{d\Lyapunovc_t}{d \Ia{t}}\right\rvert,\left\lvert\frac{d\Lyapunovc_t}{d \Ra{t}}\right\rvert,\left\lvert\frac{d\Lyapunovc_t}{d \Ib{t}}\right\rvert,\left\lvert\frac{d\Lyapunovc_t}{d \Rb{t}}\right\rvert\right) < c_s.$$
\end{lemma}

\begin{proof}
    We start by bounding $\left\lvert\frac{d\Lyapunovc_t}{d \Ia{t}}\right\rvert$. We have by \Cref{lem:derivativeL} that

    \begin{align*}
        \frac{d\Lyapunovc_t}{d \Ia{t}} & = \xa(1- \frac{\Iae}{\Ia{t}}) + 2\yb\frac{\Rbe \ln(\Iae)}{\Ia{t} \ln^2(\Ia{t})}\ln(\frac{\Rb{t}+\Rbe}{\Rbe(1+\frac{\ln(\Iae)}{\ln(\Ia{t})})}).
    \end{align*}

    As $\Rae + \Rbe < (1-c)n$, $\Iae$ is linear and thus $\sqrt{\Iae} < \varepsilon_l n \leq \Ia{t}$. Hence,

    $$1 \leq 1+\frac{\ln(\Iae)}{\ln(\Ia{t})} \leq 3.$$

    For large enough $c_s$ we get

    \begin{align*}
        \left\lvert\frac{d\Lyapunovc_t}{d \Ia{t}}\right\rvert = & \left\lvert\xa(1- \frac{\Iae}{\Ia{t}}) + 2\yb\frac{\Rbe \ln(\Iae)}{\Ia{t} \ln^2(\Ia{t})}\ln(\frac{\Rb{t}+\Rbe}{\Rbe(1+\frac{\ln(\Iae)}{\ln(\Ia{t})})})\right\rvert \\
        \leq                                                    & \xa(1+\frac{n}{\varepsilon_l n}) +2\yb \frac{n \ln(n)}{\varepsilon_l n \ln^2(\varepsilon_l n)}\ln(\frac{2n}{\Rbe})                                                 \\
        \leq                                                    & \xa(1+\frac{1}{\varepsilon_l}) +\frac{c_s}{\ln(n)}                                                                                                                 \\
        <                                                       & c_s.
    \end{align*}

    The bound for $\Ib{t}$ follows symmetrically. For $\Ra{t}$ (and symmetrical for $\Rb{t}$) we get for large enough $c_s$

    \begin{align*}
        \left\lvert\frac{d\Lyapunovc_t}{d \Ra{t}}\right\rvert = & \left\lvert2\ya(1-\frac{\Rae(1+\frac{\ln(\Ibe)}{\ln(\Ib{t})})}{\Ra{t}+\Rae})\right\rvert \\
        \leq                                                    & 2\ya (1+\frac{3n}{\Rae})                                                                 \\
        <                                                       & c_s.\qedhere
    \end{align*}
\end{proof}

We now define the drift of the process rigorously.

\begin{definition}
    Let $C$ be a \IRIR process on a graph $G$ with $n$ vertices and density $\density \in (0,1]$. Let $C$ have constant recovery rates $\ra,\rb \in \R_{>0}$ and effective infection rates $\lae=\ca/n$ and $\lbe=\cb/n$ for constants $\ca,\cb \in \R_{>0}$. Let $i \in \N$ and let $t\coloneqq \timeContinuous{i}$ and $t'\coloneqq \timeContinuous{i+1}$ such that $\Ia{t}\geq 2$ and $\Ib{t}\geq 2$. We define the \emph{drift} at time step $i$ to be
    \begin{align*}
        D_i & \coloneqq \E{\Lyapunovc_{t'}-\Lyapunovc_{t}}[\filtration_{t}].\qedhere
    \end{align*}
\end{definition}

Using the previously shown bounds we get the following lemma. It shows that the drift of the process equals the derivative normalized by the rate of change plus lower order terms.

\begin{lemma}\label{lem:driftDerivative}
    Let $C$ be a \IRIR process on a perfectly mixed graph $G$ with $n$ vertices and density $\density \in (0,1]$. Let $C$ have constant recovery rates $\ra,\rb \in \R_{>0}$ and effective infection rates $\lae=\ca/n$ and $\lbe=\cb/n$ for constants $\ca,\cb \in \R_{>0}$. Let $\varepsilon_l$ be a constant and let $i \in \N$ with $t\coloneqq \timeContinuous{i}$ such that $\Ia{t}\geq \varepsilon_l n$ and $\Ib{t}\geq \varepsilon_l n$. Let $r_t= \lae \Ia{t} \Rb{t} + \ra \Ia{t} +\lbe \Ib{t} \Ra{t} + \rb \Ib{t}$ the total rate of all clocks at time $t$. Then

    \begin{align*}
        D_i & = \frac{d \Lyapunovc_t}{d t}\cdot \frac{1}{r_t}\pm \bigO{n^{-1}}.\qedhere
    \end{align*}

\end{lemma}

\begin{proof}
    How the state of $C$ changes in the next step is decided by Poisson clocks. The probability of each of those clocks triggering next is proportional to their rates. Hence

    \begin{align*}
        D_i= & \frac{\lae \Ia{t} \Rb{t}}{r_t}\left(\Lyapunovc(\Ia{t}+1,\Ra{t},\Ib{t},\Rb{t}-1)-\Lyapunovc(\Ia{t},\Ra{t},\Ib{t},\Rb{t})\right)   \\
             & + \frac{\ra \Ia{t}}{r_t}\left(\Lyapunovc(\Ia{t}-1,\Ra{t}+1,\Ib{t},\Rb{t})-\Lyapunovc(\Ia{t},\Ra{t},\Ib{t},\Rb{t})\right)         \\
             & + \frac{\lbe \Ib{t} \Ra{t}}{r_t}\left(\Lyapunovc(\Ia{t},\Ra{t}-1,\Ib{t}+1,\Rb{t})-\Lyapunovc(\Ia{t},\Ra{t},\Ib{t},\Rb{t})\right) \\
             & + \frac{\rb \Ib{t}}{r_t}\left(\Lyapunovc(\Ia{t},\Ra{t},\Ib{t}-1,\Rb{t}+1)-\Lyapunovc(\Ia{t},\Ra{t},\Ib{t},\Rb{t})\right).
    \end{align*}

    We calculated the derivatives of $\Lyapunovc_t$ by all $I$'s and $R$'s before (see \Cref{lem:derivativeL}). From \Cref{lem:differenceDerivative} we get that increasing/decreasing one of the $I$'s or $R$'s by one changes the value of $\Lyapunovc_t$ by its derivative $\pm\bigO{n^{-1}}$. Also, as $\Ia{t}\geq \varepsilon_l n$ and $\Ib{t}\geq \varepsilon_l n$, all rates are in $\bigO{n}$, $r_t \in \bigTheta{n}$ and the derivatives are in $\bigO{1}$. Hence

    \begin{align*}
        D_i= & \frac{\lae \Ia{t} \Rb{t}-\ra \Ia{t}}{r_t}(\frac{d\Lyapunovc_t}{d\Ia{t}}\pm \bigO{n^{-1}})  \\
             & +\frac{\ra \Ia{t}-\lbe \Ib{t} \Ra{t}}{r_t}(\frac{d\Lyapunovc_t}{d\Ra{t}}\pm \bigO{n^{-1}}) \\
             & +\frac{\lbe \Ib{t} \Ra{t}-\rb \Ib{t}}{r_t}(\frac{d\Lyapunovc_t}{d\Ib{t}}\pm \bigO{n^{-1}}) \\
             & +\frac{\rb \Ib{t}-\lae \Ia{t} \Rb{t}}{r_t}(\frac{d\Lyapunovc_t}{d\Rb{t}}\pm \bigO{n^{-1}}) \\
        =    & \frac{d \Lyapunovc_t}{dt} \cdot \frac{1}{r_t}\pm \bigO{n^{-1}}.\qedhere
    \end{align*}
\end{proof}

We now bound the derivative of $\Lyapunovc$ in the region we care for. We show that it is sufficiently negative in the given region to apply a negative drift theorem later.

\begin{lemma}\label{lem:drift}
    Let $C$ be a \IRIR process on a perfectly mixed graph $G$ with $n$ vertices and density $\density \in (0,1]$. Let $C$ have constant recovery rates $\ra,\rb \in \R_{>0}$ and effective infection rates $\lae=\ca/n$ and $\lbe=\cb/n$ for constants $\ca,\cb \in \R_{>0}$. Let there be a constant $c\in \R_{>0}$ such that $\Rae + \Rbe < (1-c)n$. Let $\varepsilon_l$ be a constant and let $i \in \N$ with $t\coloneqq \timeContinuous{i}$ such that $\Ia{t}\geq \varepsilon_l n$ and $\Ib{t}\geq \varepsilon_l n$ and let there be a small enough constant $\varepsilon_h \in \R_{>0}$ such that $\Ia{t}\leq \varepsilon_h n$. Then there exists a constant $c_d\in \R_{>0}$ such that
    \begin{align*}
        \frac{d \Lyapunovc_t}{dt} & \leq -c_d\frac{n}{\ln(n)}.\qedhere
    \end{align*}
\end{lemma}

\begin{proof}
    Consider $\da{t}$ and $\db{t}$. We do a case distinction on whether both of the $\delta$'s are small or not. To this end let $\varepsilon_\delta \in \R_{>0}$ be a small enough constant that we choose later.

    First assume that $\max(\lvert\da{t}\rvert,\lvert\db{t}\rvert)\geq \varepsilon_\delta n$. We get using \Cref{lem:derivativeL} that there is a large enough constant $A\in \R_{>0}$ such that

    \begin{align*}
        \frac{d \Lyapunovc_t}{dt}  = & \frac{d \Lyapunovb_t}{dt}                                                                                                            \\
                        & +2\yb\frac{\Rbe \ln(\Iae)}{ \ln^2(\Ia{t})}\ln(\frac{\Rb{t}+\Rbe}{\Rbe(1+\frac{\ln(\Iae)}{\ln(\Ia{t})})})(\lae\Rb{t}-\ra) \\
                        & +2\ya\frac{\Rae(1-\frac{\ln(\Ibe)}{\ln(\Ib{t})})}{\Ra{t}+\Rae}(\ra\Ia{t}-\lbe\Ib{t}\Ra{t})                               \\
                        & +2\ya\frac{\Rae \ln(\Ibe)}{ \ln^2(\Ib{t})}\ln(\frac{\Ra{t}+\Rae}{\Rae(1+\frac{\ln(\Ibe)}{\ln(\Ib{t})})})(\lbe\Ra{t}-\rb) \\
                        & +2\yb\frac{\Rbe(1-\frac{\ln(\Iae)}{\ln(\Ia{t})})}{\Rb{t}+\Rbe}(\rb\Ib{t}-\lae\Ia{t}\Rb{t})                               \\
        =               & \frac{d \Lyapunovb_t}{dt}                                                                                                            \\
                        & +2\yb\frac{\Rbe \ln(\Iae)}{ \ln^2(\Ia{t})}\ln(\frac{\Rb{t}+\Rbe}{\Rbe(1+\frac{\ln(\Iae)}{\ln(\Ia{t})})})(\lae\db{t})     \\
                        & +2\ya\frac{\Rae(\ln(\Ib{t})-\ln(\Ibe))}{(\Ra{t}+\Rae)\ln(\Ib{t})}(\ra\Ia{t}-\lbe\Ib{t}\Ra{t})                            \\
                        & +2\ya\frac{\Rae \ln(\Ibe)}{ \ln^2(\Ib{t})}\ln(\frac{\Ra{t}+\Rae}{\Rae(1+\frac{\ln(\Ibe)}{\ln(\Ib{t})})})(\lbe\da{t})     \\
                        & +2\yb\frac{\Rbe(\ln(\Ia{t})-\ln(\Iae))}{(\Rb{t}+\Rbe)\ln(\Ia{t})}(\rb\Ib{t}-\lae\Ia{t}\Rb{t})                            \\
        \leq            & \frac{d \Lyapunovb_t}{dt}                                                                                                            \\
                        & +2\yb\frac{n \ln(n)}{ \ln^2(\varepsilon_l n)}\ln(\frac{2n}{\Rbe})(\ca)                                                   \\
                        & +2\ya\frac{n \ln(n/\Ibe)}{\Rae\ln(\varepsilon_l n)}(\ra n)                                                               \\
                        & +2\ya\frac{n \ln(n)}{ \ln^2(\varepsilon_l n)}\ln(\frac{2n}{\Rae})(\cb)                                                   \\
                        & +2\yb\frac{n \ln(n/\Iae)}{\Rbe\ln(\varepsilon_l n)}(\rb n)                                                               \\
        \leq            & \frac{d \Lyapunovb_t}{dt} + A \frac{n}{\ln(n)}.
    \end{align*}

    Using \Cref{lem:derivativeK} with the fact that $\max(\lvert\da{t}\rvert,\lvert\db{t}\rvert)\geq \varepsilon_\delta n$, $\Ia{t} \geq \varepsilon_l n$, and $\Ib{t} \geq \varepsilon_l n$, we get that there is a constant $a \in \R_{>0}$ such that $\frac{d \Lyapunovb_t}{dt} < -an$. That is the case as all terms in \Cref{lem:derivativeK} are negative and the term including the larger $\delta$ is linear in $n$. Plugging this into the above equation gives us

    \begin{align*}
        \frac{d \Lyapunovc_t}{dt} & \leq -an + A \frac{n}{\ln(n)} \\
                      & < -\frac{a}{2}n
    \end{align*}
    As this is even asymptotically smaller than what is claimed in the lemma, that concludes this case.

    For the second case assume $\max(\lvert\da{t}\rvert,\lvert\db{t}\rvert)< \varepsilon_\delta n$. Note that for small enough $\varepsilon_\delta$ and $\varepsilon_h$, this together with $\Ia{t} \leq \varepsilon_h n$ implies that $\Ib{t}$ is above its equilibrium by a constant factor, making $\ln(\Ib{t})- \ln(\Ibe)$ positive. At the same time, $\Ia{t}$ is below its equilibrium, hence $\ln(\Ia{t})- \ln(\Iae)$ is negative. It also holds that $(\ra\Ia{t}-\lbe\Ib{t}\Ra{t})$ is upper bounded by a negative constant and $(\rb\Ib{t}-\lae\Ia{t}\Rb{t})$ is lower bounded by a positive constant. Hence in the derivative, the two terms including these rates both end up negative. We showed that $\frac{d \Lyapunovb_t}{dt}\leq 0$ in every state. Using these facts, we get that there is a small enough constant $a \in \R_{>0}$ and a large enough constant $A \in \R_{>0}$ both independent of $\varepsilon_\delta$ such that

    \begin{align*}
        \frac{d \Lyapunovc_t}{dt}  = & \frac{d \Lyapunovb_t}{dt}                                                                                                        \\
                        & +2\yb\frac{\Rbe \ln(\Iae)}{ \ln^2(\Ia{t})}\ln(\frac{\Rb{t}+\Rbe}{\Rbe(1+\frac{\ln(\Iae)}{\ln(\Ia{t})})})(\lae\db{t}) \\
                        & +2\ya\frac{\Rae(\ln(\Ib{t})-\ln(\Ibe))}{(\Ra{t}+\Rae)\ln(\Ib{t})}(\ra\Ia{t}-\lbe\Ib{t}\Ra{t})                        \\
                        & +2\ya\frac{\Rae \ln(\Ibe)}{ \ln^2(\Ib{t})}\ln(\frac{\Ra{t}+\Rae}{\Rae(1+\frac{\ln(\Ibe)}{\ln(\Ib{t})})})(\lbe\da{t}) \\
                        & +2\yb\frac{\Rbe(\ln(\Ia{t})-\ln(\Iae))}{(\Rb{t}+\Rbe)\ln(\Ia{t})}(\rb\Ib{t}-\lae\Ia{t}\Rb{t})                        \\
        \leq            & 0                                                                                                                    \\
                        & +2\yb\frac{n \ln(n)}{ \ln^2(\varepsilon_l n)}\ln(\frac{2n}{\Rbe})(\ca \varepsilon_\delta)                            \\
                        & -2\ya\frac{\Rae (\ln(\Ib{t})-\ln(\Ibe))}{2n\ln(n)}(\lbe \Ib{t}\Ra{t}-\ra \varepsilon_h n)                            \\
                        & +2\ya\frac{n \ln(n)}{ \ln^2(\varepsilon_l n)}\ln(\frac{2n}{\Rae})(\cb \varepsilon_\delta)                            \\
                        & -2\yb\frac{\Rbe (\ln(\Iae)-\ln(\Ia{t})))}{2n\ln(n)}(\rb \Ib{t}-\lae \Rb{t}\varepsilon_h n)                           \\
        \leq            & \varepsilon_\delta A \frac{n}{\ln(n)}                                                                                \\
                        & - a \frac{n}{\ln(n)}                                                                                                 \\
                        & + \varepsilon_\delta A \frac{n}{\ln(n)}                                                                              \\
                        & - a \frac{n}{\ln(n)}                                                                                                 \\
        \leq            & - a \frac{n}{\ln(n)}.
    \end{align*}

    The last inequality holds when choosing $\varepsilon_\delta \leq \frac{a}{2A}$. This concludes the second case and therefore the proof.
\end{proof}

We now show the main technical theorem which is an application of the negative drift theorem. It gives a bound on the distribution of the survival time of an IRIR process on perfectly mixed graphs.

\begin{theorem}\label{thm:clique}
    Let $C$ be a \IRIR process on a perfectly mixed graph $G$ with $n$ vertices and density $\density \in (0,1]$. Let $C$ have constant recovery rates $\ra,\rb \in \R_{>0}$ and effective infection rates $\lae=\ca/n$ and $\lbe=\cb/n$ for constants $\ca,\cb \in \R_{>0}$. Let there be a constant $c\in \R_{>0}$ such that $\Rae + \Rbe < (1-c)n$. Let there be a small enough constant $\varepsilon_h \in \R_{>0}$ such that $\Ia{0}\geq \varepsilon_h n$ and $\Ib{0}\geq \varepsilon_h n$. Let $T_{\mathrm{d}}$ be the number of steps until the infections in $C$ die out. Then there exists a constant $c_T \in \R_{>0}$ such that for sufficiently large $n$
    \begin{align*}
        \Pr{T_{\mathrm{d}}\leq t} & \leq t^2 \cdot e^{-c_T\frac{n}{\ln(n)}}.\qedhere
    \end{align*}
\end{theorem}

\begin{proof}
    The idea is to use the negative drift theorem \Cref{pre:negativeDrift} on the potential function $\Lyapunovc$ from \Cref{def:lyapunovPotential3}. To this end we first show that there is an interval of linear length in the potential space in which all states that fall into this interval have a linear number of infected vertices but one of the infected vertices is below $\varepsilon_h n$. Those restrictions are necessary for most of our lemmas to apply. We then apply the lemmas proven in this section to show that in this region the step size is bounded by a constant and the drift is negative. This lets us apply \Cref{pre:negativeDrift} to bound the survival time.

    By \Cref{lem:potentialOneSmall} there exists a constant $c_l$ such that $\Lyapunovc_t> c_l n$ implies $\Ia{t}< \varepsilon_h n$ or $\Ib{t} < \varepsilon_h n$. Note that this means that $\Lyapunovc_0 \leq c_l n$.

    Let $\varepsilon_l \in \R_{>0}$ be a constant such that the corresponding $c_h$ from \Cref{lem:potentialBothBig} is at least $2c_l$. We then get by applying \Cref{lem:potentialBothBig} that $\Lyapunovc_t < c_h n$ implies $\Ia{t}> \varepsilon_l n$ and $\Ib{t} > \varepsilon_l n$.

    That means that as long as $\Lyapunovc_t < c_h n$, the conditions for \Cref{lem:constantStep} and \Cref{lem:differenceDerivative} are met and applying both of them together gives us that there is a constant $c_s\in \R_{>0}$ such that the change of $\Lyapunovc_t$ within one step is upper bounded by $c_s$.

    We now define the process $(X_i)_{i\in \N}$ which is the discrete version of $\Lyapunovc_t$ shifted and scaled so that we can apply the negative drift theorem. It is defined as

    \begin{align*}
        X_i & = \frac{\Lyapunovc_{\tau_{i}}-c_l n}{c_s}-1.
    \end{align*}

    We now show that all conditions of \Cref{pre:negativeDrift} are met. Let $(\filtration_i)_{i\in \N}$ be the natural filtration of $X_i$. Let $a=-1$, $b=\frac{c_h n - c_l n}{c_s}-1$, $c=1$ and let $T_{\mathrm{d}}'=\inf\{i \in \N \mid X_i \geq b\}$. First note that $\Lyapunovc_0 \leq c_l n$ implies that $X_0 \leq 0$. For all $i<T_{\mathrm{d}}'$ it holds that $X_i < b$ which implies $\Lyapunovc_{\tau_{i}} \leq c_h n$ and therefore $\Ia{\tau_{i}}> \varepsilon_l n$ and $\Ib{\tau_{i}} > \varepsilon_l n$. We use this to show that the three items are fulfilled for all $i<T_{\mathrm{d}}'$.

    For the first point note that $\E{(X_{i+1}-X_i)}[\filtration_i]= \frac{D_i}{c_s}$. For states with $X_i \geq a$ we get by \Cref{lem:drift} that there is a constant $c_d' \in \R_{>0}$ such that $\frac{d \Lyapunovc_t}{dt} \leq -c_d'\frac{n}{\ln(n)}$. Together with \Cref{lem:driftDerivative} and the fact that with $\Ia{\tau_i} \geq \varepsilon_l n$ it holds $r_{\tau_i} \in \bigTheta{n}$ we get that there is a constant $c_d \in \R_{>0}$ such that $D_i \leq -\frac{c_d}{\ln(n)}$. Hence

    \begin{align*}
        \E{(X_{i+1}-X_i)\cdot \indicator{X_i\geq a}}[\filtration_i]\leq -\frac{c_d}{\ln(n)}\cdot \indicator{X_i\geq a}.
    \end{align*}

    The second point follows directly from the fact that $\Lyapunovc_{\tau_i}$ changes by at most $c_s$ in one step for all $i<T_{\mathrm{d}}'$. Hence the step size of $X_i$ is upper bounded by $c=1$.

    This also directly implies the third point as $a=-1$, hence one step cannot change $X_i<a$ to $X_{i+1}>0$.

    Applying \Cref{pre:negativeDrift} gives us for all $t \in \N$
    \begin{align*}
        \Pr{T_{\mathrm{d}}'\leq t} & \leq t^2 \cdot e^{-\frac{(\frac{c_h n - c_l n}{c_s}-1)\frac{c_d}{\ln(n)}}{2}}.
    \end{align*}

    As for all $i<T_{\mathrm{d}}'$ holds $\Ia{\tau_{i}}> \varepsilon_l n$ and $\Ib{\tau_{i}} > \varepsilon_l n$, we know $T_{\mathrm{d}}\geq T_{\mathrm{d}}'$. Hence using the above inequality we get that there is a constant $c_T \in \R_{>0}$ such that for sufficiently large $n$

    \begin{align*}
        \Pr{T_{\mathrm{d}}\leq t} & \leq t^2 \cdot e^{-c_T\frac{n}{\ln(n)}}.\qedhere
    \end{align*}
\end{proof}

We extend this theorem to jumbled graphs. To this end we first show that the drift on jumbled graphs does only differ by the drift on cliques by lower order terms.

\begin{lemma}\label{lem:cliqueJumbled}
    Let $C$ be a \IRIR process on a perfectly mixed graph $G$ with $n$ vertices and density $\density \in (0,1]$. Let $C$ have constant recovery rates $\ra,\rb \in \R_{>0}$ and effective infection rates $\lae=\ca/n$ and $\lbe=\cb/n$ for constants $\ca,\cb \in \R_{>0}$. Let $\varepsilon_l$ be a constant and let $i \in \N$ with $t\coloneqq \timeContinuous{i}$ such that $\Ia{t}\geq \varepsilon_l n$ and $\Ib{t}\geq \varepsilon_l n$.

    Let $k \in \bigOmega{1}$ and let $C^*$ be a \IRIR process on a $(p,p n/k)$-jumbled graph $G^*$ with $n$ vertices. Let $C^*$ also have the recovery rates $\ra,\rb$ and effective infection rates $\lae, \lbe$. Furthermore, at step $i \in \N$, let $C^*$ have the same number of vertices in each state as in $C$. Then
    \begin{align*}
        D^*_i & = D_i \pm \bigO{k^{-1}}.\qedhere
    \end{align*}

\end{lemma}

\begin{proof}
    As both $C$ and $C^*$ have the same number of vertices in each state, by definition of $\Lyapunovc$ and by the choice of the model parameters, the values of $\Lyapunovc$ and $\Lyapunovc^*$ are the same. Hence the only difference in the drift are the rates at which vertices change states. The recovery rate is the same in both processes, hence the total rates of vertices recovering are also the same. Let $e^*_t(I_2,R_1)$ be the number of edges in $G^*$ from vertices in $I_2$ to vertices in $R_1$ at time $t$ (and equivalently for 1 and 2 swapped). Then in $C^*$ vertices get infected by $I_2$ at a rate of $\lbe/p \cdot e^*_t(I_2,R_1)$. As $G^*$ is $(p,p /k)$-jumbled, this can be written as $\lbe/p \cdot (p \Ib{t} \Ra{t} \pm \bigO{p  n^2/k}) = \lbe \Ib{t} \Ra{t} \pm \bigO{n/k}$. Let $r_t= \lae \Ia{t} \Rb{t} + \ra \Ia{t} +\lbe \Ib{t} \Ra{t} + \rb \Ib{t}$ the total rate of all clocks at time $t$ in $C$. Let $r^*_t$ be the equivalent rate in $C^*$. Then $r_t \in \bigTheta{n}$ and $r^*_t = r_t \pm \bigO{n/k}$. By \Cref{lem:differenceDerivative} and \Cref{lem:constantStep} the changes is the $L$ values in one step are in $\bigO{1}$. Using all of that we get

    \begin{align*}
        D^*_i= & \frac{\lae/p \cdot e^*_t(I_1,R_2)}{r^*_t}\left(\Lyapunovc^*(\Ia{t}+1,\Ra{t},\Ib{t},\Rb{t}-1)-\Lyapunovc^*(\Ia{t},\Ra{t},\Ib{t},\Rb{t})\right)                   \\
               & + \frac{\ra \Ia{t}}{r^*_t}\left(\Lyapunovc^*(\Ia{t}-1,\Ra{t}+1,\Ib{t},\Rb{t})-\Lyapunovc^*(\Ia{t},\Ra{t},\Ib{t},\Rb{t})\right)                                  \\
               & + \frac{\lbe/p \cdot e^*_t(I_2,R_1)}{r^*_t}\left(\Lyapunovc^*(\Ia{t},\Ra{t}-1,\Ib{t}+1,\Rb{t})-\Lyapunovc^*(\Ia{t},\Ra{t},\Ib{t},\Rb{t})\right)                 \\
               & + \frac{\rb \Ib{t}}{r^*_t}\left(\Lyapunovc^*(\Ia{t},\Ra{t},\Ib{t}-1,\Rb{t}+1)-\Lyapunovc^*(\Ia{t},\Ra{t},\Ib{t},\Rb{t})\right)                                  \\
        =      & \frac{\lae \Ia{t} \Rb{t} \pm \bigO{n/k}}{r_t \pm \bigO{n/k}}\left(\Lyapunovc(\Ia{t}+1,\Ra{t},\Ib{t},\Rb{t}-1)-\Lyapunovc(\Ia{t},\Ra{t},\Ib{t},\Rb{t})\right)    \\
               & + \frac{\ra \Ia{t}}{r_t \pm \bigO{n/k}}\left(\Lyapunovc(\Ia{t}-1,\Ra{t}+1,\Ib{t},\Rb{t})-\Lyapunovc(\Ia{t},\Ra{t},\Ib{t},\Rb{t})\right)                         \\
               & + \frac{\lbe \Ib{t} \Ra{t} \pm \bigO{n/k}}{r_t \pm \bigO{n/k}}\left(\Lyapunovc(\Ia{t},\Ra{t}-1,\Ib{t}+1,\Rb{t})-\Lyapunovc(\Ia{t},\Ra{t},\Ib{t},\Rb{t})\right)  \\
               & + \frac{\rb \Ib{t}}{r_t \pm \bigO{n/k}}\left(\Lyapunovc(\Ia{t},\Ra{t},\Ib{t}-1,\Rb{t}+1)-\Lyapunovc(\Ia{t},\Ra{t},\Ib{t},\Rb{t})\right)                         \\
        =      & \left(\frac{\lae \Ia{t} \Rb{t} }{r_t } \pm \bigO{k^{-1}}\right)\left(\Lyapunovc(\Ia{t}+1,\Ra{t},\Ib{t},\Rb{t}-1)-\Lyapunovc(\Ia{t},\Ra{t},\Ib{t},\Rb{t})\right) \\
               & + \left(\frac{\ra \Ia{t}}{r_t }\pm \bigO{k^{-1}}\right)\left(\Lyapunovc(\Ia{t}-1,\Ra{t}+1,\Ib{t},\Rb{t})-\Lyapunovc(\Ia{t},\Ra{t},\Ib{t},\Rb{t})\right)         \\
               & + \left(\frac{\lbe \Ib{t} \Ra{t}}{r_t}\pm \bigO{k^{-1}}\right)\left(\Lyapunovc(\Ia{t},\Ra{t}-1,\Ib{t}+1,\Rb{t})-\Lyapunovc(\Ia{t},\Ra{t},\Ib{t},\Rb{t})\right)  \\
               & + \left(\frac{\rb \Ib{t}}{r_t}\pm \bigO{k^{-1}}\right)\left(\Lyapunovc(\Ia{t},\Ra{t},\Ib{t}-1,\Rb{t}+1)-\Lyapunovc(\Ia{t},\Ra{t},\Ib{t},\Rb{t})\right)          \\
        =      & D_i \pm \bigO{k^{-1}}.\qedhere
    \end{align*}
\end{proof}

Doing some small adjustments to the proof of \Cref{thm:clique}, we get the same result for jumbled graphs instead of perfectly mixed graphs.

\begin{corollary}\label{cor:jumbled}
    Let $k \in \R_{>0}$ with $k\in \smallOmega{\ln(n)}$, $p \in [0,1]$ (possibly dependent on $n$) and let $C$ be a \IRIR process on a $(p,p n/k)$-jumbled graph $G$ with $n$ vertices. Let $C$ have constant recovery rates $\ra,\rb \in \R_{>0}$ and effective infection rates $\lae=\ca/n$ and $\lbe=\cb/n$ for constants $\ca,\cb \in \R_{>0}$. Let there be a constant $c\in \R_{>0}$ such that $\frac{\ra}{\lae} + \frac{\rb}{\lbe} < (1-c)n$. Let there be a small enough constant $\varepsilon_h \in \R_{>0}$ such that $\Ia{0}\geq \varepsilon_h n$ and $\Ib{0}\geq \varepsilon_h n$. Let $T_{\mathrm{d}}$ be the number of steps until the infections in $C$ die out. Then there exists a constant $c_T \in \R_{>0}$ such that for sufficiently large $n$
    \begin{align*}
        \Pr{T_{\mathrm{d}}\leq t} & \leq t^2 \cdot e^{-c_T\frac{n}{\ln(n)}}.\qedhere
    \end{align*}
\end{corollary}

\begin{proof}
    The proof is almost the same as the proof for \Cref{thm:clique}, so we only highlight the changes. The only change is that the drift $D_i$ is slightly different. However, in the region considered in the proof, by \Cref{lem:cliqueJumbled} it only differs from the drift of a perfectly mixed graph by $\pm \bigO{k^{-1}}$ which, as $k\in \smallOmega{\ln(n)}$, does not change the bound used in the proof which is of order $\bigTheta{\ln(n)^{-1}}$.

    With this adjustment the proof of \Cref{thm:clique} applies for $(p,p n/k)$-jumbled graphs as well, resulting in exactly the same bound (except for the $c_T$ being a slightly different constant).
\end{proof}

We now plug in some value for $t$ and transfer the result to the continuous survival time of the process.

\jumbledContinuous*

\begin{proof}
    The idea is to use \Cref{cor:jumbled} and plug in an appropriate $t$ that is large but still gives a super-polynomially small probability of the discrete survival time being smaller than $t$. We then use the fact that each step is exponentially distributed together with \Cref{lem:sumExponential} to bound the actual survival time.

    By \Cref{cor:jumbled} there exists a constant $c_T \in \R_{>0}$ such that for all $t\in \N$ holds $\Pr{T_{\mathrm{d}}\leq t} \leq t^2 \cdot e^{-c_T\frac{n}{\ln(n)}}$. Choosing $t= e^{\frac{c_T n}{4 \ln(n)}}$ gives $\Pr{T_{\mathrm{d}}\leq e^{\frac{c_T n}{4 \ln(n)}}} \leq e^{-\frac{c_T n}{2 \ln(n)}}$.

    Note that the total rate $r$ of all clocks in the system is linear in $n$. Hence, while the rate of change at each step depends on the current state, we can define an exponential clock with rate $r$ for each step that are all independent such that the time of each step dominates the corresponding clock. Therefore $T$ dominates the sum of $T_{\mathrm{d}}$ exponential random variables of rate $r$ each. Let $X$ be that sum. We get

    \begin{align*}
        \Pr{T\leq e^{\frac{c_T n}{8 \ln(n)}}} \leq & \Pr{X\leq e^{\frac{c_T n}{8 \ln(n)}}}                                                                                                                               \\
        \leq                                       & \Pr{X\leq e^{\frac{c_T n}{8 \ln(n)}}}[T_{\mathrm{d}}\leq e^{\frac{c_T n}{4 \ln(n)}}]\cdot \Pr{T_{\mathrm{d}}\leq e^{\frac{c_T n}{4 \ln(n)}}}                        \\
                                                   & + \Pr{X\leq e^{\frac{c_T n}{8 \ln(n)}}}[T_{\mathrm{d}}> e^{\frac{c_T n}{4 \ln(n)}}]\cdot \Pr{T_{\mathrm{d}}> e^{\frac{c_T n}{4 \ln(n)}}}                            \\
        \leq                                       & \Pr{T_{\mathrm{d}}\leq e^{\frac{c_T n}{4 \ln(n)}}} + \Pr{X\leq e^{\frac{c_T n}{8 \ln(n)}}}[T_{\mathrm{d}}> e^{\frac{c_T n}{4 \ln(n)}}].\numberthis\label{eq:TUpper}
    \end{align*}

    We already bounded the first term, so it is left to bound the second one. The expected value of $X$ is $T_{\mathrm{d}}/r$, hence for sufficiently large $n$, $e^{\frac{c_T n}{8 \ln(n)}}$ is much smaller than half the expectation of $T_{\mathrm{d}}$ when $T_{\mathrm{d}}> e^{\frac{c_T n}{4 \ln(n)}}$. We get using \Cref{lem:sumExponential} that conditioned on $T_{\mathrm{d}}> e^{\frac{c_T n}{4 \ln(n)}}$ it holds

    \begin{align*}
        \Pr{X\leq e^{\frac{c_T n}{8 \ln(n)}}} \leq & \Pr{X \leq T_{\mathrm{d}}/2r}               \\
        \leq                                       & e^{-T_{\mathrm{d}}(1/2 -1 - \ln(1/2))}      \\
        \leq                                       & e^{-\frac{1}{6}e^{\frac{c_T n}{4 \ln(n)}}}.
    \end{align*}

    Plugging that into \Cref{eq:TUpper} gives

    \begin{align*}
        \Pr{T\leq e^{\frac{c_T n}{8 \ln(n)}}} \leq & \Pr{T_{\mathrm{d}}\leq e^{\frac{c_T n}{4 \ln(n)}}} + \Pr{X\leq e^{\frac{c_T n}{8 \ln(n)}}}[T_{\mathrm{d}}> e^{\frac{c_T n}{4 \ln(n)}}] \\
        \leq                                       & e^{-\frac{c_T n}{2 \ln(n)}} + e^{-\frac{1}{6}e^{\frac{c_T n}{4 \ln(n)}}}                                                               \\
        \leq                                       & 2e^{-\frac{c_T n}{2\ln(n)}}.\qedhere
    \end{align*}
\end{proof}

Using the fact that \erdosRenyi graphs are almost surely $(p',\sqrt{np})$-jumbled, we get almost surely super-polynomial survival times for them.

\begin{corollary}\label{cor:ErdosSurvival}
    Let $n \in \N$, let $k \in \R_{>0}$ with $k\in \smallOmega{\ln^2(n)}$ and $p = k/n \leq 0.99$ and let $C$ be a \IRIR process on a graph $G= G(n,p)$. Let $C$ have constant recovery rates $\ra,\rb \in \R_{>0}$ and effective infection rates $\lae=\ca/n$ and $\lbe=\cb/n$ for constants $\ca,\cb \in \R_{>0}$. Let there be a constant $c\in \R_{>0}$ such that $\frac{\ra}{\lae} + \frac{\rb}{\lbe} < (1-c)n$. Let there be a small enough constant $\varepsilon_h \in \R_{>0}$ such that $\Ia{0}\geq \varepsilon_h n$ and $\Ib{0}\geq \varepsilon_h n$. Let $T$ be the time until the infections in $C$ die out. Then there exists a constant $c_T \in \R_{>0}$ such that for sufficiently large $n$ almost surely
    \begin{align*}
        T & \in \bigOmega{e^{\frac{c_T n}{8 \ln(n)}}}.\qedhere
    \end{align*}
\end{corollary}

\begin{proof}
    Let $p'$ be the density of $G$. Not that almost surely $p' \in \bigTheta{p}$. Using \Cref{lem:ErdosJumbled} we get that $G$ is almost surely $(p',\sqrt{np})$-jumbled. Doing some rearranging we get

    \begin{align*}
        \sqrt{np} & = \sqrt{k} = p \frac{\sqrt{k}}{p} = p n/\sqrt{k}.
    \end{align*}

    As almost surely $p' \in \bigTheta{p}$, $G$ is almost surely $(p',\bigTheta{p'n/\sqrt{k}})$-jumbled. The corollary statement follows from $k\in \smallOmega{\ln^2(n)}$ and applying \Cref{cor:jumbledContinuous}.
\end{proof}

We obtained a super-polynomial survival time lower bound for the case when $\frac{\ra}{\lae} + \frac{\rb}{\lbe} < (1-c)n$. In the next theorem we use an easier potential function together with the additive drift theorem to give a polynomial survival time upper bound if $\frac{\ra}{\lae} + \frac{\rb}{\lbe} > (1+c)n$, showing that this condition gives some kind of threshold behavior.

\fastDieOut*

\begin{proof}
    The idea of this proof is to define a simpler potential function than we needed for the other direction and then use an additive drift theorem (\Cref{pre:additiveDrift}) to bound the expected number of steps until one of the infection dies out. We then bound the time until the second infection also dies out separately. At last we obtain a survival time bound from the bound on the number of steps.

    Let $T_1'$ be the number of steps until one of the two infection dies out. Then for all $i \in \N_{<T_1'}$ we define the potential (defining $\ln(0)=-1$)
    \begin{align*}
        P_i & \coloneqq \frac{1}{\lae}(\ln(\Ia{\tau_i})+1) + \frac{1}{\lbe}(\ln(\Ib{\tau_i})+1).
    \end{align*}

    Let $r_i$ be the total rate of change at step $i$ and let $e(I_1,R_2)$ and $e(I_2,R_1)$ be the number of edges between vertices in these states respectively at step $i$. Note that $r_i \in \bigO{n}$ so let $c_r \in \R_{>0}$ be a constant such that $r_i < c_r n$ for all $i\in \N_{<T_1'}$. We now calculate the drift for all $i \in \N_{<T_1'}$ using $\frac{\ra}{\lae} + \frac{\rb}{\lbe} > (1+c)n$, $\Ra{\tau_i} + \Rb{\tau_i} \leq n$ and $\frac{1}{x+1} \leq \ln(x+1) - \ln(x) \leq \frac{1}{x}$ for all $x \in \N_{>0}$ (and $\frac{1}{x+1} \leq \ln(x+1) - \ln(x)$ for $x=0$ using our definition of $\ln(0)$). We get

    \begin{align*}
        \E{P_{i+1}-P_i} = & \frac{\la e(I_1,R_2)}{r_i}\left(\frac{1}{\lae}(\ln(\Ia{\tau_i}+1)- \ln(\Ia{\tau_i}))\right)      \\
                          & + \frac{\ra \Ia{\tau_i}}{r_i}\left(\frac{1}{\lae}(\ln(\Ia{\tau_i}-1)- \ln(\Ia{\tau_i}))\right)   \\
                          & + \frac{\lb e(I_2,R_1)}{r_i}\left(\frac{1}{\lbe}(\ln(\Ib{\tau_i}+1)- \ln(\Ib{\tau_i}))\right)    \\
                          & + \frac{\rb \Ib{\tau_i}}{r_i}\left(\frac{1}{\lbe}(\ln(\Ib{\tau_i}-1)- \ln(\Ib{\tau_i}))\right)   \\
        \leq              & \frac{\lae \Ia{\tau_i} \Rb{\tau_i}}{r_i}\left(\frac{1}{\lae}\cdot \frac{1}{\Ia{\tau_i}}\right)   \\
                          & + \frac{\ra \Ia{\tau_i}}{r_i}\left(\frac{1}{\lae}\cdot -\frac{1}{\Ia{\tau_i}}\right)             \\
                          & + \frac{\lbe \Ib{\tau_i} \Ra{\tau_i}}{r_i}\left(\frac{1}{\lbe}\cdot \frac{1}{\Ib{\tau_i}}\right) \\
                          & + \frac{\rb \Ib{\tau_i}}{r_i}\left(\frac{1}{\lbe}\cdot -\frac{1}{\Ib{\tau_i}}\right)             \\
        =                 & \frac{1}{r_i}\left(\Rb{\tau_i} - \frac{\ra}{\lae} + \Ra{\tau_i} - \frac{\rb}{\lbe}\right)        \\
        =                 & \frac{1}{r_i}\left(\Rb{\tau_i} + \Ra{\tau_i} - (\frac{\ra}{\lae} + \frac{\rb}{\lbe})\right)      \\
        \leq              & \frac{1}{r_i}\left(n-(1+c)n\right)                                                               \\
        =                 & -\frac{cn}{r_i}                                                                                  \\
        \leq              & -\frac{c}{c_r}.
    \end{align*}

    Furthermore note that for all $i \in \N_{<T_1'}$ holds $P_i \geq 0$ and $P_i\in \bigO{n \ln(n)}$. In particular $P_0 \in \bigO{n \ln(n)}$. We aim to apply \Cref{pre:additiveDrift}. However, the stopping time in this theorem is defined as the time until the potential hits $0$. Our stopping time might stop earlier than that. To make the drift theorem directly applicable we just redefine the potential of all states at which we stopped to be $0$. This just decreases the drift further and does not break anything else. Applying \Cref{pre:additiveDrift} we get

    \begin{align*}
        \E{T_1'} & \in \bigO{n \ln(n)}.
    \end{align*}

    Let $T_2'$ be the number of steps it takes for the second infection to die out after the first did. Let w.l.o.g. $I_1$ be the first infection to die out. Then after that no vertex can enter or leave the $I_1$ state anymore. As each vertex changes states deterministically in a circle, each vertex can change its state at most twice more before the infection dies out completely. Hence $T_2' \leq 2n$.

    While the infection did not die out yet, in each state the total rate of change $r_i$ is lower bounded by $\min(\ra,\rb)$. Hence the expected length of each step is at most $\min(\ra,\rb)^{-1}$ independent of all other steps. In particular, the time until the infection dies out dominates the sum of $(T_1'+T_2')$ independent exponential variables of expected size $\min(\ra,\rb)^{-1}$. Using Wald's equation (\Cref{lem:Wald}) we get that

    \begin{align*}
        \E{T} & \leq (T_1' + T_2')\cdot \min(\ra,\rb)^{-1} \in \bigO{n \ln(n)}.\qedhere
    \end{align*}
\end{proof}

\end{document}